\documentclass{article}

\PassOptionsToPackage{numbers,sort&compress}{natbib}
\usepackage[preprint]{neurips_2022}

\usepackage[utf8]{inputenc} % allow utf-8 input
\usepackage[T1]{fontenc}    % use 8-bit T1 fonts
\usepackage{hyperref}       % hyperlinks
\usepackage{url}            % simple URL typesetting
\usepackage{booktabs}       % professional-quality tables
\usepackage{amsfonts}       % blackboard math symbols
\usepackage{nicefrac}       % compact symbols for 1/2, etc.
\usepackage{microtype}      % microtypography
\usepackage[dvipsnames,svgnames,table]{xcolor}         % colors
\usepackage{amsmath}
\usepackage{mathtools}
\usepackage{cleveref}
\usepackage{amsthm}
\usepackage{amssymb}
\usepackage{algorithm}
\usepackage[noend]{algpseudocode}
\usepackage{comment}
\usepackage{mathtools}
\usepackage{wrapfig}
\usepackage{bm}
\usepackage[normalem]{ulem}

\usepackage[shortlabels]{enumitem}
\setenumerate[1]{leftmargin=2em, itemsep=0pt, partopsep=0pt, parsep=\parskip, topsep=0pt}
\setenumerate[2]{leftmargin=2em, itemsep=0pt, partopsep=0pt, parsep=\parskip, topsep=0pt}
\setitemize[1]{leftmargin=2em, itemsep=0pt, partopsep=0pt, parsep=\parskip, topsep=0pt}
\setitemize[2]{leftmargin=2em, itemsep=0pt, partopsep=0pt, parsep=\parskip, topsep=0pt}

\usepackage{tikz}
\usetikzlibrary{arrows.meta, shapes}

\newtheorem{theorem}{Theorem}[section]
\newtheorem{lemma}[theorem]{Lemma}

\newtheorem{property}{Property}[section]
\newtheorem{definition}{Definition}[section]
\newtheorem{example}[theorem]{Example}
\newtheorem{assumption}{Assumption}[section]
\newtheorem{proposition}{Proposition}[section]
\newtheorem{corollary}[theorem]{Corollary}

\usepackage{etoolbox}
\pretocmd{\lemma}{\crefalias{theorem}{lemma}}{}{}
\pretocmd{\corollary}{\crefalias{theorem}{corollary}}{}{}
\pretocmd{\proposition}{\crefalias{theorem}{proposition}}{}{}
\Crefname{assumption}{Assumption}{Assumptions}

\newcommand{\norm}[1]{\left\lVert#1\right\rVert}
\newcommand{\abs}[1]{\left\lvert#1\right\rvert}
\newcommand{\Hre}{H_{\mathrm{re}}}

\newcommand{\diag}{\mathrm{diag}}

\newcommand{\MPC}{\mathsf{MPC}}
\newcommand{\OPT}{\mathsf{OPT}}
\newcommand{\ALG}{\mathsf{ALG}}
\newcommand{\cost}{\mathrm{cost}}

\allowdisplaybreaks

\newcommand{\yiheng}[1]{\textcolor{blue}{[Yiheng says: #1]}}
\newcommand{\guannan}[1]{\textcolor{cyan}{[Guannan says: #1]}}
\newcommand{\yang}[1]{\textcolor{purple}{[Yang says: #1]}}
\newcommand{\tongxin}[1]{\textcolor{brown}{[Tongxin says: #1]}}

\newcommand{\tabincell}[2]{\begin{tabular}{@{}#1@{}} #2 \end{tabular}}

\title{Bounded-Regret MPC via Perturbation Analysis: Prediction Error, Constraints, and Nonlinearity}

\author{%
  Yiheng Lin \\
  California Institute of Technology \\
  Pasadena, CA, USA \\
  \texttt{yihengl@caltech.edu}
  \And 
  Yang Hu \\
  Harvard University \\
  Cambridge, MA, USA \\
  \texttt{yanghu@g.harvard.edu}
  \AND
  Guannan Qu \\
  Carnegie Mellon University \\
  Pittsburgh, PA, USA \\
  \texttt{gqu@andrew.cmu.edu}
  \And
  Tongxin Li \\
  The Chinese University of Hong Kong (Shenzhen) \\
  Shenzhen, Guangdong, China \\
  \texttt{litongxin@cuhk.edu.cn}
  \AND
  Adam Wierman \\
  California Institute of Technology \\
  Pasadena, CA, USA \\
  \texttt{adamw@caltech.edu}
}

\begin{document}

\maketitle

\renewcommand{\thefootnote}{\fnsymbol{footnote}}
%\footnotetext[2]{This work is supported by NSF Grants CNS-2146814, CPS-2136197, CNS-2106403, NGSDI-2105648, EPCN-2154171, with additional support from Amazon AWS. Yiheng Lin was supported by Kortschak Scholars program. Tongxin Li was supported by the start-up funding UDF01002773 of CUHK-Shenzhen.}
\renewcommand{\thefootnote}{\arabic{footnote}}

%\adam{I don't like the current title, the first phrase is a bit awkward.  My favorite so for are the last two in the list below.}

%Titles:

%Bounded-Regret MPC in Nonlinear Constrained Systems: A Perturbation-based Approach

%Predictive control in non-linear, constrained systems: From perturbation analysis to bounded regret 

%Bounded-Regret MPC via a Perturbation-Based Analysis: Prediction Error, Constraints, and Nonlinearity

%Bounding the Regret of Predictive Control in Nonlinear Constrained Systems: A Perturbation Approach

%Bounded-Regret MPC in Nonlinear Constrained Systems via a Perturbation-Based Analysis

\begin{abstract}
We study Model Predictive Control (MPC) and propose a general analysis pipeline to bound its dynamic regret. The pipeline first requires deriving a perturbation bound for a finite-time optimal control problem. Then, the perturbation bound is used to bound the per-step error of MPC, which leads to a bound on the dynamic regret. Thus, our pipeline reduces the study of MPC to the well-studied problem of perturbation analysis, enabling the derivation of regret bounds of MPC under a variety of settings. To demonstrate the power of our pipeline, we use it to generalize existing regret bounds on MPC in linear time-varying (LTV) systems to incorporate prediction errors on costs, dynamics, and disturbances. Further, our pipeline leads to regret bounds on MPC in systems with nonlinear dynamics and constraints. 

%Our pipeline provides a vital link between the rich literature on optimization perturbation analysis and often-challenging theoretical analysis of MPC, enabling the derivation of regret bounds under a variety of settings.
 % under nonlinear time-varying dynamics with state and actuation constraints and prediction error. Our analysis pipeline allows the derivation of generalizations of existing results on the dynamic regret of MPC for linear, time-varying systems to incorporate prediction error while also leading to the first dynamic regret bound for constrained MPC and nonlinear MPC. 
\end{abstract}

\section{Introduction}\label{sec:intro}

\begin{comment}
Main contributions:
\begin{enumerate}
    \item A refined pipeline (Sec 2) that takes in perturbation bounds and outputs the performance guarantee (dynamic regret or competitive ratio). Compared with \cite{lin2021perturbation}, the new pipeline is more general because: 1) It works for non-linear dynamics; 2) It works when the exp-decay perturbation bound does not hold globally; 3) It can even work for polynomial perturbation bounds.
    \item The first competitive ratio bound for MPC in unconstrained LTV systems when there are prediction errors on disturbances.
    \item The first dynamic regret bound for MPC in unconstrained LTV systems when there are prediction errors on dynamical matrices.
    \item The first dynamic regret bound for MPC in constrained LTV systems and non-linear LTV systems. A good initial guess of $\OPT$ trajectory is needed. Compared with \cite{xu2019exponentially} and \cite{na2020superconvergence}, our method does not require the counter-intuitive replan window $L$.
    \item We provide a simple example where exp-decay perturbation bound does not hold. We show that no online algorithm (include MPC) cannot perform well in this case.
\end{enumerate}
\end{comment}

Model Predictive Control (MPC) is an optimal control approach that solves a Finite-Time Optimal Control Problem (FTOCP) using future predictions in a receding horizon manner \cite{garcia1989model}. It is a flexible approach that is able to accommodate nonlinear and time-varying dynamics, state and actuation constraints, and general cost functions \cite{rosolia2017learning, korda2018linear, allgower2012nonlinear, falcone2007linear}. As a result, it is broadly applied in a wide spectrum of control problems, including robotics \cite{wieber2006trajectory, gu2006receding, shim2003decentralized, diedam2008online, neunert2016fast}, autonomous vehicles \cite{morgan2014model, richards2006robust, cairano2008MPC, stewart2008model, amari2008unified, Hatanaka2009explicit, cairano2013vehicle}, power systems \cite{gonzalez2019powergrid, santhosh2020windforecast, lin2012online, li2018model, shetaya2017model, parisio2014model, yaramasu2014predictive}, process control \cite{wang2016combined, clarke1988application, ellis2014tutorial}, etc.

Despite the popularity of MPC, its theoretic analysis has been quite challenging. Early works along this line focused on the stability and recursive feasibility of MPC \cite{diehl2010lyapunov, angeli2011average, angeli2016theoretical, grune2020economic}. More recently, there has been tremendous interest in providing \text{finite-time learning-theoretic} performance guarantees for MPC, such as regret and/or competitive ratio bounds \cite{yu2020competitive, yu2020power}. For example, progress has recently been made toward (i) regret analysis of MPC in linear time-invariant (LTI) systems with prediction errors on the trajectory to track \cite{zhang2021regret}, (ii) the dynamic regret and competitive ratio bounds of MPC under linear time-varying (LTV) dynamics with exact predictions \cite{lin2021perturbation}, and (iii) exponentially decaying perturbation bounds of the finite-time optimal control problem in time-varying, constrained, and non-linear systems \cite{shin2020decentralized, shin2021controllability}. Beyond MPC, providing regret and/or competitive ratio guarantees for a variety of (predictive) control policies has been a focus in recent years. Examples include RHGC \cite{li2020online, li2019online} and AFHC \cite{lin2012online, chen2015online} for online control/optimization with prediction horizons, OCO-based controllers \cite{agarwal2019online, agarwal2019logarithmic} for no-regret online control, and variations of ROBD for competitive online control without predictions \cite{goel2019beyond, shi2020online} or with delayed observations \cite{pan2021online}. In addition, regret lower bounds have been studied in known LTI systems \cite{goel2020power} and unknown LTV systems \cite{minasyan2021online}. %\adam{add citations}

%\yang{Can we replace the repeated ``and/or'' here... or do we need competitive ratio?} 
A promising analysis approach that has emerged from the literature studying MPC and, more generally, predictive control, is the use of perturbation analysis techniques, or more particularly, the use of so-called exponential decaying perturbation bounds. Such techniques underlie the results in \cite{lin2021perturbation, shin2020decentralized, shin2021controllability, zhang2021regret}. %\adam{add Lina's work to this list?}.
This research direction is particularly promising since perturbation bounds exist for FTOCP in many dynamical systems, e.g., \cite{xu2019exponentially, na2020superconvergence, shin2021exponential, lee2006continuity, fiacco1990sensitivity}, and thus it potentially allows the derivation of regret and/or competitive ratio bounds in a variety of settings.
However, to this point the approach has only yielded results in unconstrained linear systems with no prediction errors (e.g., \cite{lin2021perturbation}), and often requires adjusting MPC to include a counter-intuitively large re-planning window due to technical challenges in the analysis (e.g., \cite{xu2019exponentially, na2020superconvergence}).

Thus, though perturbation analysis techniques might seem promising, many important questions about applying them for the study of predictive control remain open. Firstly, one of the major reasons for the extensive application of MPC is its flexibility in incorporating constraints and nonlinear dynamics \cite{borrelli2017predictive}. However, none of the existing results and approaches can analyze the performance of MPC under constraints and/or nonlinear dynamics. In fact, the anlyasis of MPC under constraints or nonlinearity has long been known to be challenging because of the intractable form of cost-to-go functions and optimal solutions.
Secondly, prediction error is inevitable for real-world implementations of MPC due to unpredictable noise and model mismatch, yet the analysis of MPC subject to prediction errors is limited.
Thirdly, existing approaches analyze MPC in a case-by-case manner and, in most cases, the analysis framework is specific to the assumptions of the particular case (e.g. quadratic costs, perfect predictions, etc) in a way that does not generalize to other settings \cite{yu2020power, zhang2021regret, lin2021perturbation, xu2019exponentially, na2020superconvergence}. 

%\adam{Add an "Contributions." header to call out things more concretely}

\textbf{Contributions.} In this paper, we propose a general analysis pipeline (Section \ref{sec:pipeline}) that converts perturbation bounds for an FTOCP into dynamic regret bounds for MPC across a variety of settings. %and we illustrate the pipeline via applications to four settings: (i), (ii), (iii), (iv) \adam{fill in settings and section references}.    
More specifically, the pipeline consists of three steps (see \Cref{fig:flowchart}). In Step 1, we obtain the required perturbation bounds for the specific setting. In Step 2, as shown in \Cref{lemma:pipeline-step2}, the perturbation bounds are used to bound the \textit{per-step error}, which is defined to be the error of the MPC action against the clairvoyant optimal action (see \Cref{def:per-step-error}). In Step 3, the per-step error bound is converted to a dynamic regret bound for MPC, as shown in \Cref{thm:per-step-error-to-performance-guarantee}. The full pipeline is summarized into a \emph{Pipeline Theorm} (\Cref{thm:the-pipeline-theorem}), which directly converts perturbation bounds into bounds on the dynamic regret of MPC in general settings, including those with time-variation, prediction error, constraints, and nonlinearities.  The key technical insight that enables the pipeline is the following recursive relationship between Step 2 and Step 3 (\Cref{lemma:pipeline-step2} and \Cref{thm:per-step-error-to-performance-guarantee}): Step 2 guarantees a ``small'' per-step error $e_t$ once the current state $x_t$ of MPC is ``near'' the offline optimal trajectory ($\OPT$), while Step 3 guarantees the next state $x_{t+1}$ of MPC will be near $\OPT$ if all previous per-step errors ($\{e_\tau\}_{\tau \leq t}$) are small. Thus Step 2 and Step 3 work together to guarantee MPC states are always near $\OPT$ and thus MPC per-step errors are always small (\Cref{thm:the-pipeline-theorem}).

%.... \adam{reference the Lemma and the Theorem that make up Steps 2 and 3 and talk about the proofs to highlight the novelty}

To demonstrate the power of the proposed pipeline, we apply it to a range of settings, as summarized in \Cref{table:settings}. Our first applications are to two settings with linear time-varying (LTV) dynamics and prediction errors on (i) disturbances, \Cref{sec:unconstrained:disturbances}, and (ii) the dynamical matrices and cost functions, \Cref{sec:unconstrained:dynamics}.  The state-of-the-art results in the LTV setting are \cite{lin2021perturbation},  which requires exact knowledge of the disturbances and of the dynamics.  To the best of our knowledge, our work provides the first regret result for MPC with prediction error on the dynamics (see \Cref{thm:perturbation:unconstrained-LTV-pred-err-dynamics}), a result that enables the bounds in settings where MPC is applied to learned dynamics \cite{papadimitriou2020control}.

Our second application is to a setting with nonlinear dynamics and constraints (\Cref{sec:general}).  We show the first dynamic regret bound for MPC under state and actuation constraints in nonlinear systems with general costs (\Cref{thm:perturbation:general-system}). Very few prior results exist for MPC in this setting, even with nonlinear dynamics or constraints individually.  The most related works are \cite{xu2019exponentially}, which studies constrained MPC, and \cite{na2020superconvergence}, which studies nonlinear MPC. In both cases, a counter-intuitive re-planning window is added to MPC to facilitate the analysis, a downside that our pipeline could avoid. Besides, \cite{xu2019exponentially} and \cite{na2020superconvergence} require exact predictions of the cost functions, dynamics, and constraints for the exponential convergence property of MPC to hold, while our result can apply to more general noisy predictions. %\adam{Can we say anything about tightness/generality in comparison to these too?} 
%\adam{Related Work to still be added to intro: --Hazan LTV 
%--Pull citations related to LTV from last year's paper... 
%--Cite Gautam's papers on regret-optimal? }

\section{Preliminaries}\label{sec:preliminaries}
%The Pipeline: Bounded Regret via Perturbation Analysis}\label{sec:framework}
%\adam{should we give the pipeline/framework a name.  Also, pipeline? framework? reduction? Also, I think it's best to separate the model/algorithm description from the framework description, rather than having them as subsections of the same section. So, Section 2 is something like "Preliminaries" and Section 3 is something like "The Pipeline".  } \guannan{Low Regret Predictive Control via Perturbation Analysis: LR-PCPA?, PCPA?}

%\yang{I propose the following structure: (1) introduce the EXACT optimal control problem (true system) we want to solve; (2) introduce online control with predictions; (3) introduce how prediction error can be parameterized by $\delta, \xi$, and then define the optimization problem FTOCP; (4) establish the fundamental theorem of general perturbation-based analysis, i.e., showing the main theorem under the assumption of a general perturbation bound.}\yiheng{This framework looks great to me. I made some efforts to further simplify the notations by combining $\delta_t$ and $\xi_t$.}

In this section, we first introduce the general predictive online control problem including the settings, the objective, available information, and the predictive controller class. Then, we introduce the MPC algorithm, which is a widely-used predictive controller that we focus on in this work. Specifically, we consider a general, finite-horizon, discrete-time optimal control problem with \emph{time-varying costs, dynamics and constraints}, namely
\begin{align}\label{equ:online_control_problem}
    \min_{x_{0:T}, u_{0:T-1}} &\sum_{t = 0}^{T-1} f_t(x_t, u_t; \xi_t^*) + F_T(x_T; \xi_T^*) \nonumber\\*
    \text{ s.t. }&x_{t+1} = g_{t}(x_{t}, u_{t}; \xi_{t}^*), &\forall 0 \leq t < T,\nonumber\\*
    &s_t(x_t, u_t; \xi^*_t) \leq 0, &\forall 0 \leq t < T,\\*
    &x_0 = x(0).\nonumber
\end{align}
Here, $x_t \in \mathbb{R}^n$ is the \textit{state}, $u_t \in \mathbb{R}^m$ is the \textit{control input} or \textit{action}; $f_t$ is a time-varying \textit{stage cost} function,  $g_t$ is a time-varying \textit{dynamical} function, and $s_t$ is a time-varying \textit{constraint} function, all parameterized by a ground-truth parameter $\xi_t^*$ (unknown to an online controller); and $F_T$ is a terminal cost function parameterized by $\xi_T^*$ that regularizes the terminal state.

The offline optimal trajectory $\OPT$ is obtained by solving \eqref{equ:online_control_problem} with the full knowledge of the true parameters $\xi_{0:T}^*$. In contrast, an online controller can only observe noisy estimations of the parameters in a fixed prediction horizon to decide its current action $u_t$ at each time step $t$. For example, MPC picks $u_t$ by calculating the optimal sub-trajectory confined to the prediction horizon. The objective is to design an online controller that can compete against the offline optimal trajectory $\OPT$.
We use \textit{dynamic regret} as the performance metric, which is widely used to evaluate the performance of online controllers/algorithms in the literature of online control \cite{lin2021perturbation, yu2020competitive, zhang2021regret} and online optimization \cite{li2020online, goel2019beyond, lin2020online}.
Specifically, for a concrete problem instance $(x(0), \xi_{0:T}^*)$, let $\cost(\OPT)$ denote the total cost incurred by $\OPT$, and $\cost(\ALG)$ denote the total cost incurred by an online controller $\ALG$. The \textit{dynamic regret} is defined as the worst-case additional cost incurred by $\ALG$ against $\OPT$, i.e., $\sup_{x(0), \xi_{0:T}^*} \left(\cost(\ALG) - \cost(\OPT)\right)$.
%while the competitive ratio is defined as the worst-case ratio between the total cost of $\ALG$ and $\OPT$, i.e., $\sup_{x(0), \xi_{0:T}^*} \frac{\cost(\ALG)}{\cost(\OPT)}$. \yang{do we have to mention competitive ratio?}

The formulation in \eqref{equ:online_control_problem} is general enough to include a variety of challenging settings. In this paper, we consider three important settings to illustrate how to apply our analysis pipeline. The settings differ in (a) the form of costs, dynamics, and constraints, and (b) the quantities in the system to be predicted (i.e., parameterized by $\xi_t^*$), and the prediction error allowed. An overview of the settings is presented in Table \ref{table:settings} below.

\vspace*{-4pt}
\begin{table}[H]
  \caption{Overview of the settings considered in this paper}\label{table:settings}
  \footnotesize\centering\vspace*{-6pt}
  \begin{tabular}{c|ccc|cc}
     \specialrule{1.0pt}{0pt}{0pt}
     \textbf{Section} & \textbf{Costs} & \textbf{Dynamics} & \textbf{Constraints} & \textbf{Prediction $\bm{\xi_t}$} & \textbf{Prediction error} \\\hline
     \ref{sec:unconstrained:disturbances}  & decomposable & LTV & none & disturbance: $w_t$ & arbitrary \\\hline
     \ref{sec:unconstrained:dynamics}   & quadratic & LTV & none & \tabincell{c}{cost: $Q_t, R_t, \bar{x}_t$ \\ dynamics: $A_t, B_t$}& sufficiently small \\\hline
     \ref{sec:general} & general & \tabincell{c}{non-linear\\time-varying} & \tabincell{c}{non-linear\\stage constraint} & \tabincell{c}{cost: $f_t$\\dynamics: $g_t$\\constraints: $s_t$} & sufficiently small \\
     \specialrule{1.0pt}{0pt}{0pt}
  \end{tabular}
\end{table}
\vspace*{-10pt}

In each setting, we impose different assumptions on cost functions, dynamical systems, constraints, and properties of the predicted quantities as functions of parameter $\xi_t$. In general, we require well-defined costs, Lipschitz and uniformly controllable dynamics, and Lipschitzness of the predicted quantities with regard to $\xi_t$. For constraints, additional assumptions characterizing the active constraints along and near the optimal trajectory are imposed. Detailed definitions and statements are deferred to Appendix \ref{appendix:assumptions} and Sections \ref{sec:pipeline}, \ref{sec:unconstrained}, and \ref{sec:general}. To facilitate the statement of the pipeline, we assume the following \textit{universal properties} hold throughout the paper:
\begin{itemize}
    \item \textit{Stability of $\OPT$:} there exists a constant $D_{x^*}$ such that $\norm{x_t^*} \leq D_{x^*}$ for every state $x_t^*$ on the offline optimal trajectory $\OPT$.
    \item \textit{Lipschitz dynamics:} the ground-truth dynamical function $g_t(\cdot, \cdot; \xi_t^*)$ is Lipschitz in action; i.e., for any feasible $x_t, u_t, u'_t$, $g_t$ satisfies
    $\norm{g_t(x_t, u_t; \xi_t^*) - g_t(x_t, u_t'; \xi_t^*)} \leq L_g \norm{u_t - u_t'}.$
    \item \textit{Well-conditioned costs:} every stage cost $f_t(\cdot, \cdot; \xi_t^*)$ and the terminal cost $F_T(\cdot ; \xi_T^*)$ are nonnegative, convex, and $\ell$-smooth in $(x_t, u_t)$ and $x_T$, respectively.
\end{itemize}

\subsection{Predictive Online Control}\label{sec:setting:MPC}

While Step 3 (\Cref{thm:per-step-error-to-performance-guarantee}) in our pipeline can be generally applied to all online controllers, in the subsequent applications we focus on \textit{Model Predictive Control (MPC)}, a popular classical controller. In this subsection, we first define the available information (predictions) as well as its quality (prediction power), and how general predictive online controllers make decisions. Then, we define a useful optimization problem called FTOCP, and introduce MPC as a predictive online controller.

%In this section, we first define the prediction model for the online controller and introduce the notion of prediction error that measure the quality of the predictions. Then, we present \textit{Model Predictive Control} (MPC) as a natural candidate for the online controller. %which leverages the latest available information greedily by solving a predictive optimization problem in a receding horizon fashion.

%\textbf{Available information.} %We solve the optimal control problem in an online manner, where at each time, the online controller has access to past information and predictions of future system information (which might include cost functions, dynamics, disturbances, etc.) and decides an action based on those information. 
%The predictions of future cost functions and dynamics may contain error, and to represent the prediction errors, 
We represent the uncertainties in cost functions, dynamics, constraints, and terminal costs as function families parameterized by $\xi_t$: 
$\mathcal{F}_t \coloneqq \{ f_t(x_t, u_t; \xi_t) \mid \xi_t \in \varXi_t \}, \mathcal{G}_t \coloneqq \{ g_t(x_t, u_t; \xi_t) \mid \xi_t \in \varXi_t \},$ $\mathcal{S}_t \coloneqq \{ s_t(x_t, u_t; \xi_t) \mid \xi_t \in \varXi_t \},$ and $\mathcal{F}_T \coloneqq \{ F_T(x_T; \xi_T) \mid \xi_T \in \varXi_T \}$. The online controller knows the function families $\mathcal{F}_{0:T}$, $\mathcal{G}_{0:T-1}$, and $\mathcal{S}_{0:T-1}$ as prior knowledge, but it does not know the true parameters $\xi_{0:T}^* \in \prod_{\tau = 0}^{T} \varXi_{\tau}$. Instead, at time step $t$, the online controller has access to noisy predictions of these parameters for the future $k$ time steps (where $k$ is called the \textit{prediction horizon}), represented by $\xi_{t:t+k\mid t} \in \prod_{\tau = t}^{t+k} \varXi_{\tau}$. The parameter space $\varXi_t$ at each time step $t$ may have different dimensions.

We formally define the quality of predictions by introducing the following notion of prediction error.

%are the family of stage cost functions, dynamics, and terminal cost functions parameterized by $\xi_t$, 
%They work as valid predictions of the true cost functions and dynamics.
%The true parameter $\xi_t^* \in \varXi_t$ is unknown to the online controller. Without the loss of generality, we assume $diam(\varXi_t) \coloneqq \max_{\xi, \xi' \in \varXi_t} \norm{\xi - \xi'} \leq 1$ holds for all $t$.% \guannan{I think dim may be confusing as it usually means dimension. Here we actually mean diameter? }\yiheng{Yes, shall we change to ``diam''?}\guannan{Yea, that's better!}

\begin{definition}\label{def:pred-oracle}
    %A controller has prediction window $k$ if it receives prediction of parameters $\xi_{t:t+k\mid t}$ at time step $t$. 
    The prediction error is defined as $\rho_{t, \tau} \coloneqq \norm{\xi_{t + \tau \mid t} - \xi_{t+\tau}^*}$ for an integer $\tau \geq 0$. The power of $\tau$-step-away predictions (for parameter $\xi$) is defined as $P(\tau) \coloneqq \sum_{t = 0}^{T-\tau} \rho_{t, \tau}^2$.
\end{definition}
%
%\adam{Instead of saying "to simplify notation" talk about this in the context of the algorithm as the key step of the algorithm}\yiheng{Sure, Adam. We will add a paragraph to explain why this FTOCP is important for both the MPC algorithm and its analysis.}
%
Under this noisy prediction model, a general predictive online controller $\ALG$ decides the control action based on the current state and the latest available predictions of future parameters. We formally define the class of predictive online controllers considered in this paper  in \Cref{def:online-controller}, which includes MPC as a special case.

\begin{definition}\label{def:online-controller}
    A predictive online controller $\ALG$ is a function that takes the current state $x_t$ and the available predictions $\xi_{t:t+k\mid t}$ as inputs at time $t$ and outputs the current control action $u_t$, i.e.,
    $u_t = \ALG(x_t, \xi_{t:t+k\mid t}).$
    We use $x_0 \xrightarrow{u_0} x_1 \xrightarrow{u_1} \cdots \xrightarrow{u_{T-1}} u_T$ to denote the trajectory achieved by $\ALG$, and use $x_0 \xrightarrow{u_0^*} x_1^* \xrightarrow{u_1^*} \cdots \xrightarrow{u_{T-1}^*} u_T^*$ to denote the offline optimal trajectory $\OPT$.
\end{definition}

A core component of both the design of online controllers and our analysis is the following \textit{finite-time optimal control problem} (FTOCP). Given a time interval $[t_1, t_2]$, the FTOCP solves the optimal sub-trajectory subjected to the given initial state $z$, terminal cost $F$, and a sequence of (potentially noisy) parameters $\xi_{t_1:t_2-1}, \zeta_{t_2}$, as formalized in the following definition.

\begin{definition}\label{def:FTOCP}
    The finite-time optimal control problem (FTOCP) over the horizon $[t_1, t_2]$, with initial state $z$, parameters $\xi_{t_1: t_2-1}$ and $\zeta_{t_2}$, and terminal cost $F(\cdot; \cdot)$, is defined as
    \begin{align}\label{equ:auxiliary_control_problem}
        \iota_{t_1}^{t_2}(z, \xi_{t_1:t_2-1}, \zeta_{t_2}; F) \coloneqq \min_{y_{t_1:t_2}, v_{t_1:t_2-1}} &\sum_{t = t_1}^{t_2-1} f_t(y_t, v_t; \xi_t) + F(y_{t_2}; \zeta_{t_2})\nonumber\\*
        \text{ s.t. }&y_{t+1} = g_{t}(y_{t}, v_{t}; \xi_{t}), &\forall t_1 \leq t < t_2,\nonumber\\*
        &s_t(y_t, v_t; \xi_t) \leq 0, &\forall t_1 \leq t < t_2,\\*
        &y_{t_1} = z,\nonumber
    \end{align}
    and a corresponding optimal solution as $\psi_{t_1}^{t_2}(z, \xi_{t_1:t_2-1}, \zeta_{t_2}; F)$. We shall use the shorthand notation $\psi_{t_1}^{t_2}(z, \xi_{t_1:t_2}; F) \coloneqq \psi_{t_1}^{t_2}(z, \xi_{t_1:t_2-1}, \xi_{t_2}; F)$ when the context is clear.
\end{definition}

Note that the formulation of the FTOCP in \Cref{def:FTOCP} does not include a terminal constraint set. To compensate for this, we allow the terminal cost $F(\cdot; \zeta_{t_2})$ to take value $+\infty$ in some subset of $\mathbb{R}^n$, and $\zeta_{t_2}$ is not necessarily an element in $\varXi_{t_2}$. For example, a terminal cost function that we frequently use later is the indicator function of the terminal parameter $\zeta_{t_2}$, where $\zeta_{t_2} \in \mathbb{R}^n$. We use $\mathbb{I}$ to denote such indicator terminal cost (i.e., $\mathbb{I}(y_{t_2}; \zeta_{t_2}) = 0$ if $y_{t_2} = \zeta_{t_2}$ and $\mathbb{I}(y_{t_2}; \zeta_{t_2}) = +\infty$ otherwise).

%\textbf{Remark 2:} When the cost functions $f_t$ is non-convex or the dynamics $g_t$ is non-linear, there may be multiple primal-dual optimal solution to the FTOCP that satisfies the KKT condition. To handle this issue, we will make additional assumptions on the landscape of the total cost function around the offline optimal trajectory $\OPT$ and make sure the FTOCP is always solved in a small neighborhood of $\OPT$ where the primal-dual optimal solution is unique. \yiheng{This intuition needs to be checked carefully.}

%Note that the offline optimal trajectory (i.e., the optimal solution to problem \eqref{equ:online_control_problem}) is given by $\psi_{0}^{T}(x(0), \xi_{0:T}^*; F_T)$ under \Cref{def:FTOCP}.

%\textbf{Remark 2:} Since the cost functions can be non-convex and the dynamics can be non-linear, one cannot guarantee that the optimal solution to \eqref{equ:auxiliary_control_problem} is unique.

%\subsection{Predictive Online Control}

%An oracle of prediction window $k$ may give (probably erroneous) predictions on the system information of future $k$ steps. A general \textit{predictive online controller} $\ALG$ commits $u_t = \ALG(x_t, \xi_{t:t+k-1\mid t})$. We quantify the one-step performance of $\ALG$ by its per-step error.

%\yiheng{Explain the intuition of per-step error from cost-to-go function.} 

Finally, given the definition of the FTOCP, we are ready to formally introduce MPC. The pseudocode of this online controller is given in Algorithm \ref{alg:mpc}. Basically, at time step $t$, $\MPC_k$ solves a $k$-step predictive FTOCP using the latest available parameter predictions, and commits the first control action in the solution. When there are only fewer than $k$ steps left, $\MPC_k$ directly solves a $(T-t)$-step FTOCP at time $t$ until the end of the horizon, using the predicted real terminal cost $F_T(\cdot; \xi_{T \mid t})$. This MPC controller (and its variants) has a wide range of real-world applications. 

\vspace*{-8pt}
\begin{algorithm}[H]
\caption{Model Predictive Control ($\MPC_k$)}\label{alg:mpc}
\begin{algorithmic}[1]
\Require Specify the terminal costs $F_t$ for $k \leq t < T$.
\For{$t = 0, 1, \ldots, T-1$}
    \State $t' \gets \min\{ t+k, T \}$
    \State Observe current state $x_t$ and obtain predictions $\xi_{t:t' \mid t}$.
    \State Solve and commit control action $u_t := \psi_t^{t'}(x_t, \xi_{t:t'\mid t}; F_{t'})_{v_t}$.
\EndFor
\end{algorithmic}
\end{algorithm}

\section{The Pipeline: Bounded Regret via Perturbation Analysis}\label{sec:pipeline}

% \yiheng{
% \begin{enumerate}
%     \item General perturbation bound is what? There are two forms: (3a) is perturb xi, (3b) is perturb z. Add some explainations
%     \item The pipeline: Explain the flow. Emphasize that almost every step leverages an implication of the perturbation bound. 
%     \begin{itemize}
%         \item Step 1: obtain perturbation bound of the form (3a, 3b). 
%         \item Step 2:  Bound $e_t$. Need three smaller steps 2.1, 2.2 (5a), 2.3. 
%         \item Step 3: From $e_t$ in step 2 to dynamic regret bound. Specifically we use the following Thm 3.1: (5b) 
%         Thm 3.1: Suppose a perturbation bound of (3b) holds. Further assume 
%     \end{itemize}
%     \item Step 1 is proven in a case-by-case manner. For step 2.2, 3 there is a universal result below: which Rigorously, Assumption and Thm 3.1
% \end{enumerate}
% }

The goal of this section is to give an overview of a novel analysis pipeline that converts a perturbation bound into a bound on the dynamic regret. We begin by highlighting the form of perturbation bounds required in the pipeline, and then describe the 3-step process of applying the pipeline. In subsequent sections, we apply this pipeline to obtain new regret bounds for MPC in different settings. 

\subsection{Per-Step Error and Perturbation Bounds}

A key challenge when comparing the performance of an online controller against the offline optimal trajectory is that the online controller's state $x_t$ is different from the offline optimal state $x_t^*$ at time step $t$. Due to such discrepancy in states, we cannot simply evaluate the online controller's action $u_t$ via comparison against the offline optimal action $u_t^*$. To address this challenge, our pipeline uses the notion of per-step error (\Cref{def:per-step-error}) inspired by the performance difference lemma and its proofs in reinforcement learning (RL) \cite{lin2021perturbation}.  Specifically, we compare $u_t$ to the clairvoyant optimal action one may adopt at the same state $x_t$ if all true future parameters $\xi_{t:T}^*$ are known, which leads to the definition of \textit{per-step error} as follows.

\begin{definition}\label{def:per-step-error}
    The per-step error $e_t$ incurred by a predictive online controller $\ALG$ at time step $t$ is defined as the distance between its actual action $u_t$ and the clairvoyant optimal action, i.e.,
    \[e_t \coloneqq \norm{u_t - \psi_t^T(x_t, \xi_{t:T}^*; F_T)_{v_t}}, \text{ where }u_t = \ALG(x_t, \xi_{t:t+k\mid t}).\]
    The clairvoyant optimal trajectory starting from $x_t$ is defined as $x_{t:T\mid t}^* \coloneqq \psi_t^T(x_t, \xi_{t:T}^*; F_T)_{y_{t:T}}$.
\end{definition}

Note that the clairvoyant optimal trajectory can be viewed as being generated by an MPC controller with long enough prediction horizon and exact predictions. This notion highlights the reason why MPC can compete against the clairvoyant optimal trajectory, since the per-step error in a system controlled by $\MPC_k$ becomes
$e_t = \norm{\psi_t^{t+k}(x_t, \xi_{t:t+k\mid t}; F_{t+k})_{v_t} - \psi_t^T(x_t, \xi_{t:T}^*; F_T)_{v_t}}.$
Intuitively, the per-step error converges to zero as the prediction horizon $k$ increases and the quality of predictions improves (i.e. $\norm{\xi_{t:t+k\mid t} - \xi_{t:t+k}^*} \to 0$). %This intuition can be formalized by applying the so-called perturbation bounds studied in \cite{xu2019exponentially, na2020superconvergence, shin2020decentralized, shin2021controllability, lin2021perturbation}, which bound the impact of a perturbation on the parameters on any optimization variable in the optimal solution to the FTOCP problem \eqref{equ:auxiliary_control_problem}.
%\yiheng{Add a diagram.}
%\textbf{General perturbation bound.} %\yiheng{Need to explain what is a general perturbation bound.}

This intuition highlights the important role of perturbation bounds in comparing online controllers against (offline) clairvoyant optimal trajectories.  
As we have discussed in Section \ref{sec:intro}, many previous works \cite{xu2019exponentially, na2020superconvergence, shin2020decentralized, shin2021controllability} have established (local) decaying sensitivity/perturbation bounds for different instances of the FTOCP \eqref{equ:auxiliary_control_problem}. These bounds may take different forms, but for the application of our pipeline we require two types of perturbation bounds that are both common in the literature:
\begin{enumerate}[nosep,leftmargin=.2in,label=(\alph*)]
    \item \textit{Perturbations of the parameters $\xi_{t_1:t_2}$ given a fixed initial state $z$}:
    \begin{equation}\label{equ:perturbation-bound-fix-initial}
        \norm{\psi_{t_1}^{t_2}\left(z, \xi_{t_1:t_2}; F\right)_{v_{t_1}} - \psi_{t_1}^{t_2}\left(z, \xi_{t_1:t_2}'; F\right)_{v_{t_1}}} \leq \left(\sum_{t=t_1}^{t_2} q_1(t - t_1) \delta_t\right) \norm{z} + \sum_{t=t_1}^{t_2} q_2(t - t_1) \delta_t,
    \end{equation}
    where $\delta_t \coloneqq \norm{\xi_t - \xi_t'}$ for $t \in [t_1, t_2]$, and scalar functions $q_1$ and $q_2$ satisfy
    $\lim_{t\to\infty}q_i(t) = 0$, $\sum_{t=0}^\infty q_i(t) \leq C_i$ for constants $C_i \geq 1, i=1, 2$.
    This perturbation bound is useful in bounding the per-step error $e_t$, as we will discuss in \Cref{lemma:pipeline-step2}. %\yiheng{To do: Verify the form and make sure it is universal enough for the 4 examples later.} %Depending on the parameterization of the system and assumptions, the right hand side may take different forms.  %which requires us to study it in a case-by-case manner.
    \item \textit{Perturbation of the initial state $z$ given fixed parameters $\xi_{t_1:t_2}$}:
    \begin{equation}\label{equ:perturbation-bound-fix-parameters}
        \norm{\psi_{t_1}^{t_2}\left(z, \xi_{t_1:t_2}; F\right)_{y_t/v_t} - \psi_{t_1}^{t_2}\left(z', \xi_{t_1:t_2}; F\right)_{y_t/v_t}} \leq q_3(t - t_1) \norm{z - z'}, \text{ for }t \in [t_1, t_2],
    \end{equation}
    where the scalar function $q_3$ satisfies $\sum_{t=0}^\infty q_3(t) \leq C_3$ for some constant $C_3 \geq 1$. This bound is useful in preventing the accumulation of per-step errors $e_t$ throughout the horizon (see \Cref{thm:per-step-error-to-performance-guarantee}). Compared with \eqref{equ:perturbation-bound-fix-initial}, the right hand side of \eqref{equ:perturbation-bound-fix-parameters} has a simpler form. 
\end{enumerate}
Existing perturbation bounds usually combine the above two types (\eqref{equ:perturbation-bound-fix-initial} and \eqref{equ:perturbation-bound-fix-parameters}) into a single equation that characterizes perturbations on $z$ and $\xi_{t_1:t_2}$ simultaneously, e.g., \cite{lin2021perturbation, shin2021controllability}. Here, we decompose them into two separate types because they are used in different parts of our pipeline.  %Further, the specific decaying structure of bound \eqref{equ:perturbation-bound-fix-parameters} plays a crucial role in Step 3 of the pipeline.

\begin{comment}
\begin{equation}\label{equ:general-perturbation-bound}
    \norm{\psi_{t_1}^{t_2}\left(z, \xi_{t_1:t_2}; F\right)_{y_t} - \psi_{t_1}^{t_2}\left(z', \xi_{t_1:t_2}'; F\right)_{y_t}} \leq q(t - t_1) \norm{z - z'} + \sum_{\tau = t_1}^{t_2} q(\abs{t-\tau}) \norm{\xi_\tau - \xi_\tau'},
\end{equation}
where $q$ is a decaying scalar function that satisfies $\lim_{t \to \infty} q(t) = 0$. Intuitively, this inequality implies the impact of a perturbation on the initial state $z$ or a parameter $\xi_\tau$ on a decision variable $y_t$ decays with respect to their distance ($(t-t_1)$ or $\abs{t - \tau}$) on the time axis. It is worth noticing that the right hand side of a perturbation bound can be more complicated than \eqref{equ:general-perturbation-bound} in some cases. For example, when there are prediction errors in the dynamics, the right hand side will include additional terms like the product of $\norm{\xi_\tau - \xi_\tau'}$ and $\norm{z}$ (see Section \ref{sec:unconstrained:dynamics}).
%A key feature we need to leverage from the perturbation bounds is the decaying property with respect to distance between the location of the perturbation and the location of the decision variable on the time axis: If we perturb a parameter $\xi_t$ at time $t$ to $\xi_t'$ in the FTOCP \eqref{equ:auxiliary_control_problem}, the impact on the decision variable $y_{\tau}$ (or $v_\tau$) is exponentially small in the magnitude of the perturbation with respect to $\abs{t - \tau}$.
\end{comment}

\subsection{A 3-Step Pipeline from Perturbation Bounds to Regret}\label{sec:pipeline:flowchart}

\begin{wrapfigure}{r}{140pt}
    \vspace*{-35pt}
    \centering
    \begin{tikzpicture}
    \begin{scope}[yshift=0pt]
        \fill[fill=pink] (20pt, -30pt) -- (20pt, -36pt) -- (5pt, -36pt) -- (60pt, -50pt) -- (115pt, -36pt) -- (100pt, -36pt) -- (100pt, -30pt) -- cycle;
    \end{scope}
   
    \begin{scope}[yshift=-50pt]
        \fill[fill=pink] (20pt, -30pt) -- (20pt, -36pt) -- (5pt, -36pt) -- (60pt, -50pt) -- (115pt, -36pt) -- (100pt, -36pt) -- (100pt, -30pt) -- cycle;
    \end{scope}
    
    \begin{scope}[yshift=-100pt]
        \fill[fill=pink] (20pt, -30pt) -- (20pt, -36pt) -- (5pt, -36pt) -- (60pt, -50pt) -- (115pt, -36pt) -- (100pt, -36pt) -- (100pt, -30pt) -- cycle;
    \end{scope}
    
    \draw[draw=black, line width=1.0pt] (0pt, 0pt) rectangle (120pt, -30pt);
    \node[align=center] at (60pt, -15pt) {\textit{Step 1.} obtain perturbation\\bounds  \eqref{equ:perturbation-bound-fix-initial} \& \eqref{equ:perturbation-bound-fix-parameters}};
    
    \draw[draw=black, line width=1.0pt] (0pt, -50pt) rectangle (120pt, -80pt);
    \node[align=center] at (60pt, -65pt) {\textit{Step 2.} bound the per-step\\error $e_t$ (\Cref{lemma:pipeline-step2})};
    
    \draw[draw=black, line width=1.0pt] (0pt, -100pt) rectangle (120pt, -130pt);
    \node[align=center] at (60pt, -115pt) {\textit{Step 3.} bound dynamic\\regret (\Cref{thm:per-step-error-to-performance-guarantee})};
    
    \draw[draw=black, line width=1.0pt] (0pt, -150pt) rectangle (120pt, -165pt);
    \node[align=center] at (60pt, -157.5pt) {dynamic regret bound};
    
    \draw[draw=black, line width=0.5pt, dashed] (-5pt, -32pt) rectangle (125pt, -148pt);
    \node at (130pt, -90pt) {\rotatebox{-90}{The Pipeline Theorem \ref{thm:the-pipeline-theorem}}};
\end{tikzpicture}
    \vspace*{-15pt}
    \caption{Illustrative diagram of the 3-step pipeline from perturbation analysis to bounded regret.}\label{fig:flowchart}
    \vspace*{-20pt}
\end{wrapfigure}

An overview of the pipeline is given in Figure \ref{fig:flowchart}, which illustrates the high-level ideas of the pipeline that starts by obtaining perturbation bounds, proceeds to bound the per-step error using perturbation bounds, and finally combines the per-step error and perturbation bounds to bound the dynamic regret. In the following we describe each step in detail.

\textbf{Step 1: Obtain the perturbation bounds given in (\ref{equ:perturbation-bound-fix-initial}) and (\ref{equ:perturbation-bound-fix-parameters}).} The form of the perturbation bounds depends heavily on the specific form of the FTOCP, and thus the derivation requires case-by-case study (e.g., see Section \ref{sec:unconstrained} and Section \ref{sec:general}).
However, off-the-shelf bounds are available in most cases, as there has been a rich literature on perturbation analysis of control systems (e.g., \cite{xu2019exponentially, na2020superconvergence, shin2020decentralized, shin2021controllability, lin2021perturbation} and the references therein).  The following property summarizes precisely what is expected to be derived for bounds \eqref{equ:perturbation-bound-fix-initial} and \eqref{equ:perturbation-bound-fix-parameters} in Steps 2 and 3.

\begin{property}\label{assump:pipeline-perturbation-bounds}
Suppose there exists a positive constant $R$ such that the perturbation bound \eqref{equ:perturbation-bound-fix-initial} holds for the following specifications: with $t_1 = t$ and $t_2 = t+k$ for $t < T-k$, \eqref{equ:perturbation-bound-fix-initial} holds for $F: \mathbb{R}^n \to \mathbb{R}^n$ be the identity function $\mathbb{I}$, and
\[z \in \mathcal{B}(x_t^*, R);~ \xi_{t:t+k-1} \in \varXi_{t:t+k-1}, \xi_{t:t+k-1}' = \xi_{t:t+k-1}^*;~ \xi_{t+k}, \xi_{t+k}' \in \mathcal{B}(x_{t+k}^*, R) \subseteq \mathbb{R}^n;\]
with $t_1 = t$ and $t_2 = T$ for $t \geq T-k$, \eqref{equ:perturbation-bound-fix-initial} holds for
$z \in \mathcal{B}(x_t^*, R);~ \xi_{t:T} \in \varXi_{t:T}, \xi_{t:T} = \xi_{t:T}^*;~ F = F_T.$
Further, perturbation bound \eqref{equ:perturbation-bound-fix-parameters} holds for any $z, z' \in \mathcal{B}(x_t^*, R)$ and $\xi_{t_1:t_2} = \xi_{t_1:t_2}^*$.
\end{property}

%\sout{Recall that the indicator terminal cost function $\mathbb{I}$ is defined below \Cref{def:FTOCP}.} 
As a remark, note that for the first specification of Property \ref{assump:pipeline-perturbation-bounds} with $t_1 = t$ and $t_2 = t + k$, $\xi_{t+k}$ and $\xi_{t+k}'$ live in the state space $\mathbb{R}^n$ rather than $\varXi_{t+k}$ because they represent the target terminal state of the FTOCP solved by $\MPC_k$. Intuitively, Property \ref{assump:pipeline-perturbation-bounds} states that perturbation bounds \eqref{equ:perturbation-bound-fix-initial} and \eqref{equ:perturbation-bound-fix-parameters} hold in a small neighborhood (specifically, a ball with radius $R$) around the offline optimal trajectory $\OPT$, which is much weaker than the global exponentially decaying perturbation bounds required by previous work (e.g., \cite{lin2021perturbation}) in the following sense:
%First, we only require \eqref{equ:perturbation-bound-fix-initial} and \eqref{equ:perturbation-bound-fix-parameters} to hold within a small neighborhood around the offline optimal trajectory $\OPT$, while \cite{lin2021perturbation} requires the perturbation bound to hold globally.
(i) in the general settings where the dynamical function $g_t$ is non-linear, or where there are constraints on states and actions, one cannot hope the perturbation bound to hold globally for all possible parameters \cite{shin2021exponential,shin2021controllability,na2020superconvergence};
(ii) the decay functions $\{q_i\}_{i=1, 2, 3}$ are only required to converge to zero and satisfy $\sum_{\tau = 0}^\infty q_i(\tau) \leq C_i$, which means the exponential decay rate as in \cite{lin2021perturbation} is not necessary --- in fact, polynomial decay rates can also satisfy these properties, which greatly broadens the applicability of our pipeline.

%\yiheng{Explain what Assumption \ref{assump:pipeline-perturbation-bounds} means.}

\textbf{Step 2: Bound the per-step error $\bm{e_t}$.} The core of the analysis is to apply the perturbation bounds to bound the per-step error. For $\MPC_k$, under Property \ref{assump:pipeline-perturbation-bounds}, this step can be done in a universal way, as summarized in \Cref{lemma:pipeline-step2} below. A complete proof of \Cref{lemma:pipeline-step2} can be found in \Cref{appendix:lemma-pipeline-step2}.
\begin{lemma}\label{lemma:pipeline-step2}
Let Property \ref{assump:pipeline-perturbation-bounds} hold. Suppose the current state $x_t$ satisfies $x_t \in \mathcal{B}(x_t^*, {R}/{C_3})$ and the terminal cost $F_{t+k}$ of $\MPC_k$ is set to be the indicator function of some state $\bar{y}(\xi_{t+k\mid t})$ that satisfies $\bar{y}(\xi_{t+k\mid t}) \in \mathcal{B}(x_{t+k}^*, R)$ for $t < T - k$. Then, the per-step error of $\MPC_k$ is bounded by
\begin{equation}\label{lemma:pipeline-step2:conclusion}
    e_t \leq \sum_{\tau = 0}^{k}\left(\left(\frac{R}{C_3} + D_{x^*}\right) \cdot q_1(\tau) + q_2(\tau)\right)\rho_{t, \tau} + 2R\left(\left(\frac{R}{C_3} + D_{x^*}\right) \cdot q_1(k) + q_2(k)\right).
\end{equation}
\end{lemma}

\Cref{lemma:pipeline-step2} is a straight-forward implication of perturbation bound \eqref{equ:perturbation-bound-fix-initial} specified in Property \ref{assump:pipeline-perturbation-bounds}. To see this, for $t < T - k$, note that the per-step error $e_t$ can be bounded by
\begin{subequations}\label{lemma:pipeline-step2:e0}
\begin{align}
    e_t &= \norm{\psi_t^{t+k}(x_t, \xi_{t:t+k-1\mid t}, \bar{y}(\xi_{t+k\mid t}); \mathbb{I})_{v_t} - \psi_t^T(x_t, \xi_{t:T}^*; F_T)_{v_t}} \label{lemma:pipeline-step2:e0:s0}\\
    &= \norm{\psi_t^{t+k}(x_t, \xi_{t:t+k-1\mid t}, \bar{y}(\xi_{t+k\mid t}); \mathbb{I})_{v_t} - \psi_t^{t+k}(x_t, \xi_{t:t+k-1}^*, x_{t+k\mid t}^*; \mathbb{I})_{v_t}} \label{lemma:pipeline-step2:e0:s1}\\
    &\leq \sum_{\tau = 0}^{k-1} \big( \norm{x_t} \cdot q_1(\tau) + q_2(\tau) \big) \rho_{t, \tau} + \big( \norm{x_t} \cdot q_1(k) + q_2(k)\big)\norm{\bar{y}(\xi_{t+k\mid t}) - x_{t+k\mid t}^*}.\label{lemma:pipeline-step2:e0:s2}
\end{align}
\end{subequations}
Here, we apply the principle of optimality to conclude that the optimal trajectory from $x_t$ to $x_{t+k\mid t}^*$ (i.e., $\psi_t^{t+k}(x_t, \xi_{t:t+k-1}^*, x_{t+k\mid t}^*; \mathbb{I})$ in \eqref{lemma:pipeline-step2:e0:s1}) is a sub-trajectory of the clairvoyant optimal trajectory from $x_t$ (i.e., $\psi_t^T(x_t, \xi_{t:T}^*; F_T)$ in \eqref{lemma:pipeline-step2:e0:s0}), and \eqref{lemma:pipeline-step2:e0:s2} is obtained by directly applying perturbation bound \eqref{equ:perturbation-bound-fix-initial}. Note that $\norm{x_t} \leq \frac{R}{C_3} + D_{x^*}$, and that both $\bar{y}(\xi_{t+k\mid t})$ and $x_{t+k\mid t}^*$ are in $\mathcal{B}(x_{t+k}^*; R)$ by assumption and by perturbation bound \eqref{equ:perturbation-bound-fix-parameters} specified in Property \ref{assump:pipeline-perturbation-bounds}, we conclude that \eqref{lemma:pipeline-step2:conclusion} hold for $t < T - k$. The case $t \geq T - k$ can be shown similarly. We defer the detailed proof to \Cref{appendix:lemma-pipeline-step2}.

%\yiheng{To do: Explain why \Cref{lemma:pipeline-step2} is a straight-forward implication of \eqref{equ:perturbation-bound-fix-initial}.} 

%\yang{I would prefer if this is called ``Fact'' or ``Provable Fact'' or something like that (there is an environment named ``fact'' now). ``Assumption'' might confuse readers to regard them as true assumptions.} \adam{Just making sure that you're referring to Assumption 3.2?  I also think it's weird to call it an assumption.  Maybe call it a ``Property''?}

\textbf{Step 3: Bound the dynamic regret by $\bm{\sum_{t=0}^{T-1} e_t^2}$.} This final step builds upon perturbation bound \eqref{equ:perturbation-bound-fix-parameters}, and aims at deriving dynamic regret bounds in a universal way, as stated in \Cref{thm:per-step-error-to-performance-guarantee} below. Specifically, under the assumption that a local decaying perturbation bound in the form of \eqref{equ:perturbation-bound-fix-parameters} holds around the offline optimal trajectory $\OPT$, and the property that per-step errors $e_t$ are sufficiently small, we can show that the online controller will not leave the ``safe region'' near the offline optimal trajectory as specified in Property \ref{assump:pipeline-perturbation-bounds}, and thus the dynamic regret of $\ALG$ is bounded as in \eqref{thm:per-step-error-to-performance-guarantee:dynamic-regret} (note that $\ALG$ is not confined to MPC, but is allowed to be any algorithm with bounded per-step errors). A complete proof of \Cref{thm:per-step-error-to-performance-guarantee} can be found in Appendix \ref{appendix:per-step-error-to-performance-guarantee}.

\begin{lemma}\label{thm:per-step-error-to-performance-guarantee}
Let Property \ref{assump:pipeline-perturbation-bounds} hold. If the per-step errors of $\ALG$ satisfy $e_\tau \leq {R}/{(C_3^2 L_g)}$ for all time steps $\tau < t$, the trajectory of $\ALG$ will remain close to $\OPT$ at time $t$, i.e. $x_t \in \mathcal{B}(x_t^*, {R}/{C_3})$.
%\begin{equation}\label{thm:per-step-error-to-performance-guarantee:stable}
%    x_t \in \mathcal{B}\left(x_t^*, \frac{R}{C_3}\right) \subseteq \mathcal{B}(x_t^*, R), \text{ for } t = 1, \ldots, T.
%\end{equation}
Further, if $e_t \leq {R}/{(C_3^2 L_g)}$ for all $t < T$, the dynamic regret of $\ALG$ is upper bounded by
\begin{equation}\label{thm:per-step-error-to-performance-guarantee:dynamic-regret}
    \cost(\ALG) - \cost(\OPT) = O\left(\sqrt{\cost(\OPT)\cdot \sum_{t=0}^{T-1} e_t^2} + \sum_{t=0}^{T-1} e_t^2\right).
\end{equation}
\end{lemma}

%The theorem above highlights that Specifically, step 3 obtains the following bound on the dynamic regret:
%\begin{equation}\label{equ:pipeline-regret bound}
%    \cost(\ALG) - \cost(\OPT) = O\left(\sqrt{\cost(\OPT) \cdot \sum_{t=0}^{T-1}e_t^2} + \sum_{t=0}^{T-1} e_t^2\right).
%\end{equation}

%\yiheng{To do: Explain that when the prediction horizon $k$ is sufficiently large and the prediction errors $\rho_{t, \tau}$ are sufficiently small, \Cref{lemma:pipeline-step2} and \Cref{thm:per-step-error-to-performance-guarantee} can work together to make sure that $\MPC_k$ never leaves a $R/C_3$-ball around the offline optimal trajectory.}

\textbf{Summary.} Combining Steps 2 and 3 of the pipeline yields the following \textit{Pipeline Theorem} for $\MPC_k$ (see \Cref{thm:the-pipeline-theorem}). Basically it states that, when the prediction horizon $k$ is sufficiently large and the prediction errors $\rho_{t, \tau}$ are sufficiently small, \Cref{lemma:pipeline-step2} and \Cref{thm:per-step-error-to-performance-guarantee} can work together to make sure that $\MPC_k$ never leaves a $(R/C_3)$-ball around the offline optimal trajectory $\OPT$; thus we obtain a dynamic regret bound.

\begin{theorem}[The Pipeline Theorem]\label{thm:the-pipeline-theorem}
Let Property \ref{assump:pipeline-perturbation-bounds} hold. Suppose the terminal cost $F_{t+k}$ of $\MPC_k$ is set to be the indicator function of some state $\bar{y}(\xi_{t+k\mid t})$ that satisfies $\bar{y}(\xi_{t+k\mid t}) \in \mathcal{B}(x_{t+k}^*, R)$ for all time steps $t < T - k$. Further, suppose the prediction errors $\rho_{t, \tau}$ are sufficiently small and the prediction horizon $k$ is sufficiently large, such that
\[\sum_{\tau = 0}^{k}\left(\left(\frac{R}{C_3} + D_{x^*}\right) \cdot q_1(\tau) + q_2(\tau)\right)\rho_{t, \tau} + 2R\left(\left(\frac{R}{C_3} + D_{x^*}\right) \cdot q_1(k) + q_2(k)\right) \leq \frac{R}{C_3^2 L_g}.\]
Then, the trajectory of $\MPC_k$ will remain close to $\OPT$, i.e. $x_t \in \mathcal{B}(x_t^*, {R}/{C_3})$ for all time steps $t$, and the dynamic regret of $\MPC_k$ is upper bounded by
\begin{equation}\label{thm:the-pipeline-theorem:statement}
    \cost(\MPC_k) - \cost(\OPT) = O\left(\sqrt{\cost(\OPT) \cdot E} + E\right),
\end{equation}
where $E \coloneqq \sum_{\tau = 0}^{k-1}\left(q_1(\tau) + q_2(\tau)\right)P(\tau) + \left(q_1(k)^2 + q_2(k)^2\right) T$.
\end{theorem}

%\adam{The following is moved from earlier and needs adjustment}
%\yiheng{Emphasize this result is significant.}
The proof of \Cref{thm:the-pipeline-theorem} can be found in \Cref{appendix:the-pipeline-theorem}. To interpret the dynamic regret bound in \eqref{thm:the-pipeline-theorem:statement}, note that we have $\cost(\OPT) = O(T)$ as a result of our model assumptions. Thus, the dynamic regret of $\ALG$ is in the order of $\sqrt{T E} + E$. When there is no prediction error, the regret bound $O((q_1(k) + q_2(k))\cdot T)$ reproduces the result in \cite{lin2021perturbation}, and the bound will degrade as the prediction error increases. It is also worth noticing that, when the prediction power improves over time as the online controller learns the system better and $k = \Omega(\ln{T})$, the dynamic regret can be $o(T)$. %In this case, one can obtain a sub-linear dynamic regret (see more discussions in Section \ref{sec:unconstrained:dynamics}). 

\section{Unconstrained LTV Systems}\label{sec:unconstrained}
%\yiheng{To do: Change the contents below according to the new pipeline.}

We now illustrate the use of the Pipeline Theorem by applying it in the context of (unconstrained) LTV systems with prediction errors, either on disturbances or the dynamical matrices. 

\subsection{Prediction Errors on Disturbances}\label{sec:unconstrained:disturbances}

In this section, we consider the following special case of problem \eqref{equ:online_control_problem}, where the dynamics is LTV and the prediction error can only occur on the disturbances $w_t$:
\begin{align}\label{equ:online_control_problem:unconstrained-LTV-disturbance}
    \min_{x_{0:T}, u_{0:T-1}} &\sum_{t = 0}^{T-1} \left(f_t^x(x_t) + f_t^u(u_t)\right) + F_T(x_T)\nonumber\\*
    \text{ s.t. }&x_{t+1} = A_{t} x_{t} + B_{t} u_{t} + w_{t}(\xi_t^*), &\forall 0 \leq t < T,\\*
    &x_0 = x(0).\nonumber
\end{align}
%Compared with the general problem \eqref{equ:online_control_problem}, we additionally require that the dynamical system is time-varying linear and unconstrained (i.e., $\mathcal{S}_t = \mathbb{R}^{n} \times \mathbb{R}^m$), the stage cost $f_t$ can be decomposed the sum of the state cost $f_t^x$ and control cost $f_t^u$ in \eqref{equ:online_control_problem:unconstrained-LTV-disturbance}, and the only prediction errors are on the disturbances. 
All necessary assumptions on the system are summarized below in Assumption \ref{assump:unconstrained:disturbances}.

\begin{assumption}\label{assump:unconstrained:disturbances}
Assume the following holds for the online control problem instance \eqref{equ:online_control_problem:unconstrained-LTV-disturbance}:
\begin{itemize}[nosep,leftmargin=.2in]
    \item \textit{Cost functions:} $\{f_t^x\}_{t=0}^{T-1}, \{f_t^u\}_{t=0}^{T-1}, F_T$ are nonnegative $\mu$-strongly convex and $\ell$-smooth. And we assume $f_t^x(0) = f_t^u(0) = F_T(0) = 0$ without the loss of generality.
    \item \textit{Dynamical systems:} the LTV system $\{A_t, B_t\}$ is $\sigma$-uniform controllable with controllability index $d$, and $\norm{A_t} \leq a,~ \norm{B_t} \leq b,~ \text{and}~\Vert B_t^\dagger\Vert \leq b'$ hold for all $t$, where $B_t^\dagger$ denotes the Moore–Penrose inverse of matrix $B_t$.. The detailed definitions can be found in Assumption \ref{assump:unconstrained-LTV-pred-err-disturbance} in Appendix \ref{appendix:unconstrained-LTV-disturbances}.
    \item \textit{Predicted quantities:} $\norm{w_t(\xi_t)} \leq D_w$ holds for all $\xi_t \in \varXi_t$ and all $t$. For every time step $t$, $w_t(\xi_t)$ is a $L_w$-Lipschitz function in $\xi_t$, i.e.,
    $\norm{w_t(\xi_t) - w_t(\xi_t')} \leq L_w \norm{\xi_t - \xi_t'}, \forall \xi_t, \xi_t' \in \varXi_t.$
\end{itemize}
\end{assumption}

Under Assumption \ref{assump:unconstrained:disturbances}, we can again apply the perturbation bounds shown in \cite{lin2021perturbation} to show Property \ref{assump:pipeline-perturbation-bounds}.  In particular, we already know that for some constants $H_1 \geq 1$ and $\lambda_1 \in (0, 1)$, perturbation bounds \eqref{equ:perturbation-bound-fix-initial} and \eqref{equ:perturbation-bound-fix-parameters} hold globally for $q_1(t) = 0$, $q_2(t) = H_1 \lambda_1^t$, and $q_3(t) = H_1 \lambda_1^t$. Since both of these perturbation bounds hold globally, radius $R$ in Property \ref{assump:pipeline-perturbation-bounds} can be set arbitrarily, and we shall take $R \coloneqq \max\left\{ D_{x^*}, \frac{2L_g H_1^3}{(1 - \lambda_1)^3} \right\}$ so that \Cref{thm:the-pipeline-theorem} can be applied to $\MPC_k$ with terminal cost $F_{t+k}(\cdot; \xi_{t\mid t+k}) \equiv \mathbb{I}(\cdot;0)$.  This leads to the following dynamic regret bound:

\begin{theorem}\label{thm:perturbation:unconstrained-LTV-pred-err-disturbance}
In the unconstrained LTV setting \eqref{equ:online_control_problem:unconstrained-LTV-disturbance}, under Assumption \ref{assump:unconstrained:disturbances}, when the prediction horizon $k$ is sufficiently large such that $k \geq \ln\left(\frac{4 H_1^3 L_g}{(1 - \lambda_1)^2}\right)/\ln(1/\lambda_1)$, the dynamic regret of $\MPC_k$ (Algorithm \ref{alg:mpc}) with terminal cost $F_{t+k}(\cdot; \xi_{t\mid t+k}) \equiv \mathbb{I}(\cdot; 0)$ is bounded by
$\cost(\MPC_k) - \cost(\OPT) \leq O\left(\sqrt{T \cdot \sum_{\tau = 0}^{k-1} \lambda_1^\tau P(\tau) + \lambda_1^{2k} T^2} + \sum_{\tau = 0}^{k-1} \lambda_1^\tau P(\tau)\right).$
\end{theorem}

A complete proof of \Cref{thm:perturbation:unconstrained-LTV-pred-err-disturbance} can be found in Appendix \ref{appendix:unconstrained-LTV-disturbances}. When there are no prediction errors, the bound in \Cref{thm:perturbation:unconstrained-LTV-pred-err-disturbance} reduces to $O(\lambda_1^k T)$, which reproduces the result of \cite{lin2021perturbation}. Further, it is also worth noticing that due to the form of discounted sum $\sum_{\tau=0}^{k-1} \lambda_1^{\tau} P(\tau)$, prediction errors for the near future matter more than those for the far future.

\subsection{Prediction Error on Costs and Dynamical Matrices}\label{sec:unconstrained:dynamics}

We now consider prediction errors on cost functions and dynamics, rather than disturbances. Specifically, we consider the following instance of problem \eqref{equ:online_control_problem}:
\begin{align}\label{equ:online_control_problem:unconstrained-LTV-dynamics}
    \min_{x_{0:T}, u_{0:T-1}} &\sum_{t = 0}^{T-1} \left( (x_t - \bar{x}_t(\xi_t^*))^\top Q_t(\xi_t^*) (x_t - \bar{x}_t(\xi_t^*)) + u_t^{\top} R_t(\xi_t^*) u_t \right) + F_T(x_T; \xi_t^*)\nonumber\\*
    \text{ s.t. }&x_{t+1} = A_t(\xi_{t}^*)\cdot x_{t} + B_t(\xi_{t}^*)\cdot u_{t} + w_t(\xi_{t}^*), \hspace{6em}\forall 0 \leq t < T,\\*
    &x_0 = x(0), \nonumber
\end{align}
where the terminal cost is given by $F_T(x_T; \xi_T^*) \coloneqq (x_T - \bar{x}_T(\xi_T^*))^\top P_T(\xi_T^*) (x_T - \bar{x}(\xi_T^*))$.

All necessary assumptions on the system are summarized below in Assumption \ref{assump:unconstrained:dynamics}.

\begin{assumption}\label{assump:unconstrained:dynamics}
Assume the following holds for the online control problem instance \eqref{equ:online_control_problem:unconstrained-LTV-dynamics}:
\begin{itemize}[nosep,leftmargin=.2in]
    \item \textit{Cost:} $\mu I \preceq Q_t(\xi_t) \preceq \ell I, \mu I \preceq R_t(\xi_t) \preceq \ell I,$ and $\mu I \preceq P_T(\xi_T) \preceq \ell I$, $\forall \xi_t \in \varXi_t,$ $
    \forall t$.
    \item \textit{Dynamical systems:} both the ground-truth LTV system $\{A_t(\xi_t^*), B_t(\xi_t^*)\}_{t=0}^{T-1}$ and any predicted LTV system $\{A_t(\xi_{t+\tau\mid t}), B_t(\xi_{t+\tau\mid t})\}_{\tau = 0}^{k-1}$ (for all $\xi_t \in \varXi_t$ and all $t$) satisfy the controllability assumptions in Assumption \ref{assump:unconstrained-LTV-pred-err-dynamics} in Appendix \ref{appendix:unconstrained-LTV-dynamics}.
    \item \textit{Predicted quantities:} bounds $\norm{w_t(\xi_t)} \leq D_w, \norm{\bar{x}_t(\xi_t)} \leq D_{\bar{x}}, \norm{A_t(\xi_t)} \leq a, \norm{B_t(\xi_t)} \leq b$ hold for all $\xi_t \in \varXi_t$ and all $t$. $L_A$ is a uniform Lipschitz constant such that
    $\norm{A_t(\xi_t) - A_t(\xi_t')} \leq L_A \norm{\xi_t - \xi_t'}, \forall \xi_t, \xi_t' \in \varXi_t$ holds for all $t$, and $L_B, L_Q, L_R, L_{\bar{x}}, L_w$ are defined similarly.
\end{itemize}
\end{assumption}

Under Assumption \ref{assump:unconstrained:dynamics}, we can show that for some constants $H_2 \geq 1$ and $\lambda_2 \in (0, 1)$, perturbation bounds \eqref{equ:perturbation-bound-fix-initial} and \eqref{equ:perturbation-bound-fix-parameters} hold globally for $q_1(t) = H_2 \lambda_2^{2t}$, $q_2(t) = H_2 \lambda_2^t$, and $q_3(t) = H_2 \lambda_2^t$ under the specifications of Property \ref{assump:pipeline-perturbation-bounds}. Thus, Property \ref{assump:pipeline-perturbation-bounds} holds for arbitrary $R$, and we can set $R = D_x^* + D_{\bar{x}}$ so that \Cref{thm:the-pipeline-theorem} can be applied to $\MPC_k$ with terminal cost $F_{t+k}(\cdot; \xi_{t\mid t+k}) = \mathbb{I}(\cdot ;\bar{x}(\xi_{t\mid t+k}))$, which leads to the following dynamic regret bound:

\begin{theorem}\label{thm:perturbation:unconstrained-LTV-pred-err-dynamics}
In the unconstrained LTV setting \eqref{equ:online_control_problem:unconstrained-LTV-dynamics}, under Assumption \ref{assump:unconstrained:dynamics}, when the prediction horizon $k \geq O(1)$ \footnote{When we say $z \geq O(1)$, we mean there exists $c = O(1)$ such that $z \geq c$ holds.} and the prediction errors satisfy $\sum_{\tau=0}^k \lambda_2^{2\tau} \rho_{t, \tau} \leq \Omega(1)$ for all $t$, the dynamic regret of $\MPC_k$ (Algorithm \ref{alg:mpc}) with terminal cost $F_{t+k}(\cdot; \xi_{t\mid t+k}) = \mathbb{I}(\cdot ;\bar{x}(\xi_{t\mid t+k}))$ is bounded by
$\cost(\MPC_k) - \cost(\OPT) \leq O\left(\sqrt{T \cdot \sum_{\tau = 0}^{k-1} \lambda_2^\tau P(\tau) + \lambda_2^{2k} T^2} + \sum_{\tau = 0}^{k-1} \lambda_2^\tau P(\tau)\right).$
\end{theorem}

%\adam{Anything we want to say about this result?}

The exact constants and a complete proof of \Cref{thm:perturbation:unconstrained-LTV-pred-err-dynamics} can be found in \Cref{appendix:unconstrained-LTV-dynamics}. Compared with \Cref{thm:perturbation:unconstrained-LTV-pred-err-disturbance}, \Cref{thm:perturbation:unconstrained-LTV-pred-err-dynamics} additionally requires the discounted total prediction errors $\sum_{\tau=0}^k \lambda_2^{2\tau} \rho_{t, \tau}$ to be less than or equal to some constant. This is actually expected, and emphasizes the critical difference between the prediction errors on dynamical matrices $(A_t, B_t)$ and the prediction errors on $w_t$, since an online controller cannot even stabilize the system when the predictions on $(A_t, B_t)$ can be arbitrarily bad. It is worth noting that Assumption \ref{assump:unconstrained:dynamics} requires the uniform controllability to hold for the unknown ground-truth LTV dynamics and any predicted dynamics. The goal is to ensure the perturbation bounds for KKT matrix inverse hold in \Cref{thm:diff-inverse-sophisticated}. Intuitively, this assumption is necessary because otherwise the solution of MPC (by solving FTOCP induced by the predicted dynamics) can be unbounded. We provided two examples (Example \ref{example:inverted-pendulum} and \ref{example:frequency-regulation}) that satisfy Assumption \ref{assump:unconstrained:dynamics} while the true dynamics are unknown.

\section{General Dynamical Systems}\label{sec:general}
We now move beyond unconstrained linear systems to constrained nonlinear systems given by the general online control problem \eqref{equ:online_control_problem} in Section \ref{sec:preliminaries}. All necessary assumptions are summarized in  \Cref{assump:general-dynamical-system} in \Cref{appendix:general}. Perhaps surprisingly, decaying perturbation bounds can hold even in this case.  In particular, using Theorem 4.5 in \cite{shin2021exponential}, we can show that  there exists a small constant $R$ such that, for some constants $H_3 \geq 1$ and $\lambda_3 \in (0, 1)$, perturbation bounds \eqref{equ:perturbation-bound-fix-initial} and \eqref{equ:perturbation-bound-fix-parameters} hold for $q_1(t) = 0$, $q_2(t) = H_3 \lambda_3^t$, and $q_3(t) = H_3 \lambda_3^t$.  Thus, Property \ref{assump:pipeline-perturbation-bounds} holds (see \Cref{appendix:general} for formal statements) and we can apply \Cref{thm:the-pipeline-theorem} to obtain the following dynamic regret bound:

\begin{theorem}\label{thm:perturbation:general-system}
In the general system \eqref{equ:online_control_problem}, under \Cref{assump:general-dynamical-system} in \Cref{appendix:general}, Property \ref{assump:pipeline-perturbation-bounds} holds for some positive constant $R$ and $q_1(t) = 0$, $q_2(t) = H_3 \lambda_3^t$, and $q_3(t) = H_3 \lambda_3^t$. Suppose the terminal cost $F_{t+k}$ of $\MPC_k$ is set to be the indicator function of some state $\bar{y}(\xi_{t+k\mid t})$ that satisfies $\bar{y}(\xi_{t+k\mid t}) \in \mathcal{B}(x_{t+k}^*, R)$ for $t < T - k$. Suppose the prediction errors $\rho_{t, \tau}$ are sufficiently small and the prediction horizon $k$ is sufficiently large such that
$H_3\sum_{\tau = 0}^{k-1}\lambda_3^\tau \rho_{t, \tau} + 2R H_3 \lambda_3^k \leq \frac{(1 - \lambda_3)^2 R}{H_3^2 L_g}.$
Then, the dynamic regret of $\MPC_k$ is upper bounded by
$\cost(\MPC_k) - \cost(\OPT) \leq O\left(\sqrt{T \cdot \sum_{\tau = 0}^{k-1} \lambda_3^\tau P(\tau) + \lambda_3^{2k} T^2} + \sum_{\tau = 0}^{k-1} \lambda_3^\tau P(\tau)\right).$
\end{theorem}

A complete proof of \Cref{thm:perturbation:general-system} can be found in \Cref{appendix:general}. An assumption in \Cref{thm:perturbation:general-system} that is difficult to satisfy in general is that the reference terminal states $\bar{y}(\xi_{t+k\mid t})$ of $\MPC_k$ must be close enough to the offline optimal state $x_{t+k}^*$, i.e., $\bar{y}(\xi_{t+k\mid t}) \in \mathcal{B}(x_{t+k}^*, R)$, while the offline optimal state $x_{t+k}^*$ is generally unknown. This can be achieved in some special cases, for example, when we know $\norm{\xi_t^*}$ is sufficiently small. In this case, one can first solve FTOCP $\psi_0^T\left(x_0, \mathbf{0}; F_T\right)$ and use it as a reference to set the terminal states of $\MPC_k$. This intuition is formally shown in \Cref{appendix:general}. Another limitation is that \Cref{thm:perturbation:general-system} is only a bound on the cost of $\MPC$, not its feasibility. There are many ways to guarantee recursive feasibility of $\MPC$ \cite{borrelli2017predictive}, which we leave as future work. We also discuss how to verify Assumption \ref{assump:general-dynamical-system} in two simple examples that arise from a simple inventory dynamics in Appendix \ref{appendix:inventory-control}. The first positive example shows that Assumption \ref{assump:general-dynamical-system} is not vacuous, and the second negative example shows exponentially decaying perturbation bounds may not hold when Assumption \ref{assump:general-dynamical-system} is not satisfied.

\bibliographystyle{unsrt}
\bibliography{main.bib}

%%%%%%%%%%%%%%%%%%%%%%%%%%%%%%%%%%%%%%%%%%%%%%%%%%%%%%%%%%%%
\clearpage

%%%%%%%%%%%%%%%%%%%%%%%%%%%%%%%%%%%%%%%%%%%%%%%%%%%%%%%%%%%%
%\section*{Checklist}

\begin{comment}
%%% BEGIN INSTRUCTIONS %%%
The checklist follows the references.  Please
read the checklist guidelines carefully for information on how to answer these
questions.  For each question, change the default \answerTODO{} to \answerYes{},
\answerNo{}, or \answerNA{}.  You are strongly encouraged to include a {\bf
justification to your answer}, either by referencing the appropriate section of
your paper or providing a brief inline description.  For example:
\begin{itemize}
  \item Did you include the license to the code and datasets? \answerYes{See Section~\ref{gen_inst}.}
  \item Did you include the license to the code and datasets? \answerNo{The code and the data are proprietary.}
  \item Did you include the license to the code and datasets? \answerNA{}
\end{itemize}
Please do not modify the questions and only use the provided macros for your
answers.  Note that the Checklist section does not count towards the page
limit.  In your paper, please delete this instructions block and only keep the
Checklist section heading above along with the questions/answers below.
%%% END INSTRUCTIONS %%%
\end{comment}

\appendix

%\section{Appendix}

%Optionally include extra information (complete proofs, additional experiments and plots) in the appendix.
%This section will often be part of the supplemental material.

\section{Notation Summary}\label{appendix:notations}
In this paper, we use $\alpha_{t_1:t_2}$ ($t_2 \geq t_1$) to denote a sequence of vectors $(\alpha_{t_1}, \alpha_{t_1 + 1}, \ldots, \alpha_{t_2})$. For ease of reference, we summarize in the following table all the notations used in the paper.

\vspace*{-10pt}
\begin{table}[H]
  \centering
  \begin{tabular}{c|l}
    \specialrule{1.0pt}{0pt}{0pt}
    \textbf{Notation} & \hspace*{12.5em}\textbf{Meaning} \\\hline
    $\xi_t$ & \tabincell{l}{The uncertainty parameter of the system, used to parameterize costs,\\dynamics, and constraints.} \\\hline
    $\xi_t^*$ & The ground-truth parameter of the system, unknown to the controller. \\\hline
    $\xi_{\tau\mid t}$ & The prediction of $\xi_\tau^*$ revealed to the controller at time step $t$ ($\tau \geq t$). \\\hline
    $\varXi_t$ & \tabincell{l}{The space of uncertainty parameters. $\xi_t^*$ and $\xi_{t\mid \tau}, \tau \leq t$ are in $\varXi_t$. We \\ assume the diameter of $\varXi_t$ is less than or equal to $1$ without the loss of\\ generality, i.e., $\norm{\xi_t - \xi_t'} \leq 1$ for all $\xi_t, \xi_t' \in \varXi_t$.} \\\hline
    $k$ & \tabincell{l}{The prediction horizon. At time $t$, the controller observes predictions\\ $\xi_{t:t'\mid t}$, where $t' \coloneqq \min\{t+k, T\}$.}\\\hline
    $\rho_{t, \tau}$ & \tabincell{l}{The error of predicting the system parameter after $\tau$ steps at time $t$,\\ i.e., $\rho_{t, \tau} = \norm{\xi_{t+\tau}^* - \xi_{t+\tau\mid t}}$. We adopt the convention that  $\rho_{t, \tau} \coloneqq 0$\\ if $t+\tau > T$.} \\\hline
    $P(\tau)$ & \tabincell{l}{The total error of predicting the system parameter after $\tau$ steps (the\\power of $\tau$-step-away predictions), i.e., $P(\tau) \coloneqq \sum_{t=0}^{T-\tau} \rho_{t, \tau}^2$.} \\\hline
    $f_t(x_t, u_t; \xi_t)$ & \tabincell{l}{The stage cost of FTOCP at time step $t$, parameterized by $\xi_t \in \varXi_t$.\\The true stage cost is $f_t(x_t, u_t; \xi_t^*)$.} \\\hline
    $g_t(x_t, u_t; \xi_t)$ & \tabincell{l}{The dynamical function at time step $t$, parameterized by $\xi_t \in \varXi_t$.\\The true dynamics is $x_{t+1} = g_t(x_t, u_t; \xi_t^*)$.} \\\hline
    $s_t(x_t, u_t; \xi_t)$ & \tabincell{l}{The constraint function at time step $t$, parameterized by $\xi_t \in \varXi_t$.\\The true constraint is $s_t(x_t, u_t; \xi_t^*) \leq 0$.} \\\hline
    $F_T$ and $\{F_{t+k}\}_{t=0}^{T-k-1}$ & \tabincell{l}{$F_T$ is the true terminal cost function defined by the original online\\ control problem \eqref{equ:online_control_problem}, while $F_{t+k}$ for $t < T - k$ is the terminal cost\\ function used by $\MPC_k$ at time $t$.} \\\hline
    $\iota_{t_1}^{t_2}(z,\xi_{t_1:t_2-1}, \zeta_{t_2}; F)$ & \tabincell{l}{The FTOCP defined on the time interval $[t_1, t_2]$, where $z$ is the initial\\
    state at time $t_1$, and $F$ is some terminal cost function at time $t_2$.\\ $\xi_{t_1:t_2-1}$ are the parameters for the cost, dynamics, and constraints at\\ time $[t_1, t_2-1]$, while
    $\zeta_{t_2}$ is the parameter for the terminal cost $F$.} \\\hline
    $\psi_{t_1}^{t_2}(z,\xi_{t_1:t_2-1}, \zeta_{t_2}; F)$ & \tabincell{l}{An optimal solution to the FTOCP $\iota_{t_1}^{t_2}(z,\xi_{t_1:t_2-1}, \zeta_{t_2}; F)$. The entries\\ are indexed by $y_{t_1:t_2}$ (for states) and $v_{t_1:t_2-1}$ (for actions).}\\\hline
    $\psi_{t_1}^{t_2}(z,\xi_{t_1:t_2}; F)$ & The shorthand notation of $\psi_{t_1}^{t_2}(z,\xi_{t_1:t_2-1}, \xi_{t_2}; F)$. \\
    \specialrule{1.0pt}{0pt}{0pt}
  \end{tabular}
\end{table}

\section{Assumptions Overview}\label{appendix:assumptions}
In this section, we give a more detailed overview of the assumptions that the online control problem \eqref{equ:online_control_problem} should satisfy in general so that our pipeline in Section \ref{sec:pipeline:flowchart} works. Specific assumptions in each specific setting will be presented separately in Assumption \ref{assump:unconstrained-LTV-pred-err-disturbance}, \ref{assump:unconstrained-LTV-pred-err-dynamics}, and \ref{assump:general-dynamical-system}.

\textbf{Cost functions.} In general, we require the stage cost functions $f_t$ and the terminal cost $F_T$ to be \textit{well-conditioned}, which includes non-negativity, strong convexity, smoothness (Lipschitz continuous gradient), and twice continuous differentiability. Note that these assumptions are equivalent to bounded Hessian ($\mu I \preceq \nabla^2 f_t \preceq \ell I$) and non-negative minimizer of the cost functions. Specifically, for quadratic costs $\nabla^2 f_t$ are constant, and the assumptions are further equivalent to bounded spectra of the cost matrices.

% For general cost functions, we may also assume a \textit{bounded mismatch} along the trajectory formed by minimizers of stage costs (see \Cref{assump:unconstrained-LTV-pred-err-dynamics}), so that it is possible for the controller to minimize the total cost by injecting control actions of mild magnitude. 

\textbf{Dynamical systems.} A basic requirement of the dynamical function $g_t$ is \textit{Lipschitzness} in $u_t$, i.e.,
\[\norm{g_t(x_t, u_t; \xi_t^*) - g_t(x_t, u_t'; \xi_t^*)} \leq L_g \norm{u_t - u_t'}.\]
We point out that only Lipschitzness in control action $u_t$ is needed for the Pipeline Theorem to hold, which guarantees that an error on a control action $u_t$ has a bounded impact on the next state $x_{t+1}$.

A more non-trivial assumption on dynamics is that the dynamical system should be \textit{(uniformly) controllable}. Intuitively, this means the online controller should be able to steer the system to some target state in a finite number of time steps with some bounded control actions. 

\begin{definition}[uniform controllability]\label{def:controllability}
Consider a general dynamics $x_{t+1} = g_t(x_t, u_t; \xi_t)$. For any time steps $t_2 \geq t_1$ and fixed $(x_t, u_t)$, define $A_t := \nabla_{x_t}^{\top} g_t(x_t, u_t; \xi_t)$ and $B_t := \nabla_{u_t}^{\top} g_t(x_t, u_t; \xi_t)$, and we further define \textbf{transition matrix} $\varPhi(t_2, t_1) \in \mathbb{R}^{n \times n}$ at $(x_t, u_t)$ as
\[\varPhi(t_2, t_1) := \begin{cases}
    A_{t_2 - 1} A_{t_2 - 2} \cdots A_{t_1} & \text{ if }t_2 > t_1,\\
    I & \text{ otherwise.}
\end{cases}.\]
For any time $t$ and time interval $p \geq 0$, define \textbf{controllability matrix} $M(t, p; x_{t:t+p}, u_{t:t+p}) \in \mathbb{R}^{n \times (mp)}$ as
\[M(t, p; x_{t:t+p}, u_{t:t+p}) := \left[\varPhi(t+p, t+1) B_t, \varPhi(t+p, t+2) B_{t+1}, \ldots, \varPhi(t+p, t+p) B_{t+p}\right].\]
We say the system is \textbf{controllable} if there exists a positive integer $d$, such that the controllability matrix $M(t, d; x_t, u_t)$ is of full row rank for any $t$ and any $(x_t, u_t)$. The smallest such constant $d$ is called the \textbf{controllability index} of the system. Further, we say the system is \textbf{$\bm{\sigma}$-uniformly controllable} if exists a positive constant $\sigma$ such that $\sigma_{\min} \left( M(t, d) \right) \geq \sigma$ holds for all $t = 0, \ldots, T - d$.
\end{definition}

The definition has a clear control-theoretic interpretation for linear dynamics (where $A_t$ and $B_t$ are independent of $(x_t, u_t)$), but might seem trickier for non-linear dynamics (where $A_t$ and $B_t$ are functions of $(x_t, u_t)$). For the latter case, uniform controllability may be assumed for the offline optimal trajectory only, or for state-action pairs in a small neighborhood around it.

\textbf{Constraints.} Recursive feasibility is a well-known challenge for the design of online controllers in constrained systems \cite{borrelli2017predictive}: at some time $t$, the controller may encounter an absence of feasible trajectories to continue from the current state $x_t$.
Many solutions have been proposed for different controllers in a variety of systems.
%For example, when the online controller is MPC and the system is linear time-varying with constraints and without disturbances, the MPC controller can stay feasible if the terminal cost is set to be the indicator of the origin \footnote{In this paper, we let an indicator function $\mathbb{I}_p(x)$ take value $0$ if $x = p$ and $+\infty$ otherwise.} \yiheng{add cites}. 
Since the purpose of this work is to establish dynamic regret guarantees for an online controller, and for the purpose of this paper, we would expect that there is a solution, potentially via a combination of proper controller design (e.g., setting the terminal cost/constraint of MPC) and some additional assumptions on the system (e.g., the SSOSC, strong second-order sufficient conditions, and LICQ, linear independent constraint qualification, which will be introduced in Section \ref{sec:general}), so that we could focus on the sub-optimality of the online controller against the offline optimal trajectory. %\yang{I still think we need to explicitly talk about the terminology like uSOSC and uLICQ. The feasibility issue is actually settled by these specific assumptions.}\yiheng{Let's keep it simple here.} \adam{Maybe mention the names uSOSC and uLICQ as examples (in an e.g.) with forward pointers?}
%\yang{not very good intuition: we need uSOSC or uLICQ that prevent bad active constraints appearing together}

We also need to point out that, although the additional assumptions on system that involve constraints might seem tricky, sometimes they are exactly the implications of previous assumptions on costs and dynamics that is actually needed in the proof. For example, Lemma 12 in \cite{shin2021controllability} shows that Lipshitzness of dynamics and uniform controllability together imply uniform LICQ property of the system. For the clarity of exposition, these implications might be directly assumed in place of the low-level ones.

\textbf{Parameter $\bm{\xi_t}$.} %In our formulation, we adopt a general scheme where all uncertainties in the cost functions and dynamical systems are parameterized jointly by a single parameter $\xi_t$. %To see why this is general, consider a linear system $x_{t+1} = A_t x_t + B_t u_t + w_t$ where we know in advance that all possible $A_t, B_t,$ and $w_t$ are in some bounded sets $\mathcal{A}_t, \mathcal{B}_t,$ and $\mathcal{W}_t$ respectively. 
%As we discussed before, the ground-truth of these parameters are not available to the online controller, and the controller can only observe noisy predictions of the true parameters to decide the control actions. 
In general, we require that all predicted quantities, which might include cost functions, dynamical functions, and constraints, should be \textit{Lipschitz} in $\xi_t$, so that these quantities get closer to their ground truth value in a linearly-bounded way as the prediction error on the parameter $\xi_t^*$ decreases. For a specific example of parameterized linear dynamics $x_{t+1} = A_t(\xi_t) x_t + B_t(\xi_t) u_t + w_t(\xi_t)$, the requirement is realized by assuming Lipschitzness of $A_t(\cdot), B_t(\cdot), w_t(\cdot)$ in $\xi_t$.
%In our formulation, multiple quantities might be parameterized by $\xi$. As has been explained above, this means the ground-truth of these quantities are not available to the controller, so that the controller must generate predicted quantities themselves. For a quantity $X_t(\xi_t)$ parameterized by $\xi_t$, we require \textit{Lipschitzness} in $\xi_t$, i.e., $\norm{X_t(\xi_t) - X_t(\xi'_t)} \leq L_X \norm{\xi_t - \xi'_t}$ for some constant $L_X$.

\textbf{Offline optimal trajectory.} We require the offline optimal trajectory $\OPT$ to be \textit{stable}; i.e., there exists a constant $D_{x^*}$ such that $\norm{x_t^*} \leq D_{x^*}$ for any state $x_t^*$ visited by $\OPT$. While this can be shown under some assumptions in unconstrained LTV systems (see \cite{lin2021perturbation}), we introduce this assumption to simplify and unify the presentation for more complex systems.

\section{Proof of Lemma \ref{lemma:pipeline-step2}}\label{appendix:lemma-pipeline-step2}
We have already shown \eqref{lemma:pipeline-step2:conclusion} holds for all time step $t < T - k$ in the main body. For $t \geq T - k$, we see that
\begin{subequations}\label{lemma:pipeline-step2:e1}
\begin{align}
    e_t &= \norm{\psi_t^T\left(x_t, \xi_{t:T\mid t}; F_T\right) - \psi_t^T\left(x_t, \xi_{t:T}^*; F_T\right)}\label{lemma:pipeline-step2:e1:s1}\\
    &\leq \sum_{\tau = 0}^{k} \big( \norm{x_t} \cdot q_1(\tau) + q_2(\tau) \big) \rho_{t, \tau} \label{lemma:pipeline-step2:e1:s2}\\
    &\leq \sum_{\tau = 0}^{k} \big( \left(\frac{R}{C_3} + D_{x^*}\right) \cdot q_1(\tau) + q_2(\tau) \big) \rho_{t, \tau}, \label{lemma:pipeline-step2:e1:s3}
\end{align}
\end{subequations}
where we used the definition of per-step error $e_t$ in \eqref{lemma:pipeline-step2:e1:s1}; we used the perturbation bound \eqref{equ:perturbation-bound-fix-initial} specified by Property \ref{assump:pipeline-perturbation-bounds} in \eqref{lemma:pipeline-step2:e1:s2}; we used the assumption $x_t \in \mathcal{B}\left(x_t^*, \frac{R}{C_3}\right)$, $\norm{x_t^*} \leq D_{x^*}$, and the convention $\rho_{t, \tau} \coloneqq 0$ if $t + \tau > T$ in \eqref{lemma:pipeline-step2:e1:s3}. Thus $e_t$ also satisfies \eqref{lemma:pipeline-step2:conclusion} for $t \geq T - k$.

\section{Proof of Lemma \ref{thm:per-step-error-to-performance-guarantee}}\label{appendix:per-step-error-to-performance-guarantee}
To simplify the notation, we will use $\psi_t^T(z)$ as a shorthand notation of  $\psi_t^T(z, \xi_{t:T}^*; F_T)$ in the proof of \Cref{thm:per-step-error-to-performance-guarantee}, since the proof only relies on the perturbation bound \eqref{equ:perturbation-bound-fix-parameters}.

Note that for any time step $t+1$, by Lipschitzness of the dynamics we have
\begin{align}\label{thm:oracle-perf-bound:e0}
    \norm{x_{t+1} - \psi_t^T\left(x_t\right)_{y_{t+1}}} &= \norm{g_t(x_t, u_t, w_t) - g_t\left(x_t, \psi_t^T(x_t)_{v_t}, w_t\right)}\nonumber\\
    &\leq L_g \norm{u_t - \psi_t^T(x_t)_{v_t}}\nonumber\\
    &\leq L_g e_t.
\end{align}

Therefore, we can show the statement that $x_t \in \mathcal{B}\left(x_t^*, \frac{R}{C_3}\right)$ holds if $e_\tau \leq R/(C_3^2 L_g), \forall \tau < t$ by induction. Note that this statement clearly holds for $t = 0$ since $x_0^* = x_0$. Suppose it holds for $0, 1, \ldots, t-1$. Then, we see that
\begin{subequations}\label{thm:oracle-perf-bound:e1}
\begin{align}
    \norm{x_t - x_t^*}
    ={}& \norm{x_t - \psi_0^T(x_0)_{y_t}}\nonumber{}\\
    \leq{}& \norm{x_t - \psi_{t-1}^{T}(x_{t-1})_{y_{t}}} + \sum_{i=1}^{t-1} \norm{\psi_{t-i}^{T}(x_{t-i})_{y_t} - \psi_{t-i-1}^{T}(x_{t-i-1})_{y_{t}}}\nonumber\\
    \leq{}& \norm{x_t - \psi_{t-1}^{T}(x_{t-1})_{y_{t}}} + \sum_{i = 1}^{t-1} q_3(i) \norm{x_{t-i} - \psi_{t-i-1}^{T}(x_{t-i-1})_{y_{t-i}}}\label{thm:oracle-perf-bound:e1:s1}\\
    \leq{}& \sum_{i = 0}^{t-1} q_3(i) \norm{x_{t-i} - \psi_{t-i-1}^{T}(x_{t-i-1})_{y_{t-i}}}\label{thm:oracle-perf-bound:e1:s2}\\
    \leq{}& L_g \sum_{i=0}^{t-1} q_3(i) e_{t-i-1},\label{thm:oracle-perf-bound:e1:s3}
\end{align}
\end{subequations}
where in \eqref{thm:oracle-perf-bound:e1:s1}, we apply the perturbation bound \eqref{equ:perturbation-bound-fix-parameters} specified by Property \ref{assump:pipeline-perturbation-bounds}. To see why it can be applied, note that for $i \in [1, t-1]$, $x_{t-i-1}$ satisfies $x_{t-i-1} \in \mathcal{B}\left(x_{t-i-1}^*, \frac{R}{C_3}\right)$ by the induction assumption, thus we have $\psi_{t-i-1}^T(x_{t-i-1})_{y_{t-i}} \in \mathcal{B}\left(x_{t-i}^*, R\right)$ because $q_3(1) \leq \sum_{\tau=0}^\infty q_3(\tau) \leq C_3$. Therefore, we can apply the perturbation bound \eqref{equ:perturbation-bound-fix-parameters} specified by Property \ref{assump:pipeline-perturbation-bounds} to compare the optimization solution vectors $\psi_{t-i}^T(x_{t-i})$ and $\psi_{t-i}^T\left(\psi_{t-i-1}^T(x_{t-i-1})_{y_{t-i}}\right),$ and by the principle of optimality, we see that
\[\psi_{t-i}^T\left(\psi_{t-i-1}^T(x_{t-i-1})_{y_{t-i}}\right)_{y_t} = \psi_{t-i-1}^{T}(x_{t-i-1})_{y_{t}}.\]
We also used $q_3(0) \geq 1$ in \eqref{thm:oracle-perf-bound:e1:s2} and \eqref{thm:oracle-perf-bound:e0} in \eqref{thm:oracle-perf-bound:e1:s3}. Recall that we assume $e_{t-i} \leq \frac{R}{C_3^2 L_g}$. Substituting this into \eqref{thm:oracle-perf-bound:e1} gives that
\[\norm{x_t - x_t^*} \leq L_g \cdot \frac{R}{C_3^2 L_g} \sum_{i=0}^{t-1} q_3(i) \leq \frac{R}{C_3}.\]
Hence we have shown $x_t \in \mathcal{B}\left(x_t^*, \frac{R}{C_3}\right)$ holds if $e_\tau \leq R/(C_3^2 L_g), \forall \tau < t$ by induction. An implication of this result is that $x_t \in \mathcal{B}\left(x_t^*, \frac{R}{C_3}\right)$ holds for all $t \leq T$ if $e_t \leq R/(C_3^2 L_g)$ holds for all $t < T$.

Similar with \eqref{thm:oracle-perf-bound:e1}, we see the following inequality holds for all $t < T$ if $e_t \leq R/(C_3^2 L_g), \forall t < T$:
\begin{align}\label{thm:oracle-perf-bound:e2}
    \norm{u_t - u_t^*}
    ={}& \norm{u_t - \psi_0^T(x_0)_{v_t}}\nonumber\\
    \leq{}& \norm{u_t - \psi_{t}^{T}(x_{t})_{v_{t}}} + \sum_{i=0}^{t-1} \norm{\psi_{t-i}^{T}(x_{t-i})_{v_t} - \psi_{t-i-1}^{T}(x_{t-i-1})_{v_{t}}}\nonumber\\
    \leq{}& \norm{u_t - \psi_{t}^{T}(x_{t})_{v_{t}}} + \sum_{i = 0}^{t-1} q_3(i) \norm{x_{t-i} - \psi_{t-i-1}^{T}(x_{t-i-1})_{y_{t-i}}}\nonumber\\
    \leq{}& e_t + L_g \sum_{i=0}^{t-1} q_3(i) e_{t-i-1},
\end{align}
where the second inequality holds for the same reason as \eqref{thm:oracle-perf-bound:e1:s1}.

By \eqref{thm:oracle-perf-bound:e1}, we see that
\begin{subequations}\label{thm:oracle-perf-bound:e3}
\begin{align}
    \norm{x_t - x_t^*}^2 &\leq L_g^2 \left(\sum_{i=0}^{t-1} q_3(i) e_{t-i-1}\right)^2 \nonumber\\
    &\leq L_g^2 \left(\sum_{i=0}^{t-1} q_3(i)\right)\cdot \left(\sum_{i=0}^{t-1} q_3(i) e_{t-i-1}^2\right) \label{thm:oracle-perf-bound:e3:s1}\\
    &\leq C_3 L_g^2 \left(\sum_{i=0}^{t-1} q_3(i) e_{t-i-1}^2\right), \label{thm:oracle-perf-bound:e3:s2}
\end{align}
\end{subequations}
where we use the Cauchy-Schwarz inequality in \eqref{thm:oracle-perf-bound:e3:s1}, and $\sum_{i=0}^{t-1} q_3(i) \leq C_3$ in \eqref{thm:oracle-perf-bound:e3:s2}.

Similarly, by \eqref{thm:oracle-perf-bound:e2}, we see that
\begin{subequations}\label{thm:oracle-perf-bound:e4}
\begin{align}
    \norm{u_t - u_t^*}^2 &\leq \left(e_t + L_g \sum_{i=0}^{t-1} q_3(i) e_{t-i-1}\right)^2\nonumber\\
    &\leq \left(1 + L_g^2 \sum_{i=0}^{t-1} q_3(i)\right)\cdot \left(e_t^2 + \sum_{i=0}^{t-1} q_3(i) e_{t-i-1}^2\right)\label{thm:oracle-perf-bound:e4:s1}\\
    &\leq \left(1 + C_3 L_g^2\right)\cdot \left(e_t^2 + \sum_{i=0}^{t-1} q_3(i) e_{t-i-1}^2\right),\label{thm:oracle-perf-bound:e4:s2}
\end{align}
\end{subequations}
where we use the Cauchy-Schwarz inequality in \eqref{thm:oracle-perf-bound:e4:s1}, and we use $\sum_{i=0}^{t-1} q_3(i) \leq C_3$ in \eqref{thm:oracle-perf-bound:e4:s2}.

Summing \eqref{thm:oracle-perf-bound:e3} and \eqref{thm:oracle-perf-bound:e4} over time steps $t$ gives that
\begin{align}\label{thm:oracle-perf-bound:e5}
    &\sum_{t=1}^T \norm{x_t - x_t^*}^2 + \sum_{t=0}^{T-1} \norm{u_t - u_t^*}^2\nonumber\\
    \leq{}& C_3 L_g^2 \sum_{t=1}^T \left(\sum_{i=0}^{t-1} q_3(i) e_{t-i-1}^2\right) + \left(1 + C_3 L_g^2\right)\cdot \sum_{t=0}^{T-1} \left(e_t^2 + \sum_{i=0}^{t-1} q_3(i) e_{t-i-1}^2\right)\nonumber\\
    \leq{}& \left(1 + 2 C_3 L_g^2\right)\cdot \left(1 + C_3\right)\cdot \sum_{t=0}^{T-1} e_t^2,
\end{align}
where we rearrange the terms and use $\sum_{j=0}^\infty q_3(j) \leq C_3$ in the last inequality.

Since the cost function $f_t(\cdot, \cdot; \xi_t^*)$ and $F_T(\cdot; \xi_T^*)$ are nonnegative, convex, and $\ell$-smooth in their inputs, by Lemma F.2 in \cite{lin2021perturbation}, we see that the following inequality holds for arbitrary $\eta > 0$:
\begin{subequations}\label{thm:oracle-perf-bound:e6}
\begin{align}
    &\cost(\ALG) - \cost(\OPT)\nonumber\\
    \leq{}&\left(\sum_{t=0}^{T-1} f_t(x_t, u_t; \xi_t^*) + F_T(x_T; \xi_T^*)\right) - \left(\sum_{t=0}^{T-1} f_t(x_t^*, u_t^*; \xi_t^*) + F_T(x_T^*; \xi_T^*)\right)\nonumber\\
    \leq{}& \eta \left(\sum_{t=0}^{T-1} f_t(x_t^*, u_t^*; \xi_t^*) + F_T(x_T^*; \xi_T^*)\right)\nonumber\\
    &+ \frac{\ell}{2}\left(1 + \frac{1}{\eta}\right)\left(\sum_{t=1}^T \norm{x_t - x_t^*}^2 + \sum_{t=0}^{T-1} \norm{u_t - u_t^*}^2\right)\label{thm:oracle-perf-bound:e6:s0}\\
    \leq{}& \eta \cdot \cost(\OPT) + \left(1 + \frac{1}{\eta}\right)\cdot \frac{\ell}{2} \cdot \left(1 + 2 C_3 L_g^2\right)\cdot \left(1 + C_3\right)\cdot \sum_{t=0}^{T-1} e_t^2 \label{thm:oracle-perf-bound:e6:s1}\\
    ={}& \eta \cdot \cost(\OPT) + \frac{1}{\eta}\cdot \frac{\ell}{2} \cdot \left(1 + 2 C_3 L_g^2\right)\cdot \left(1 + C_3\right)\cdot \sum_{t=0}^{T-1} e_t^2\nonumber\\
    &+ \frac{\ell}{2} \cdot \left(1 + 2 C_3 L_g^2\right)\cdot \left(1 + C_3\right)\cdot \sum_{t=0}^{T-1} e_t^2, \label{thm:oracle-perf-bound:e6:s2}
\end{align}
\end{subequations}
where we apply Lemma F.2 in \cite{lin2021perturbation} in \eqref{thm:oracle-perf-bound:e6:s0}, and we use \eqref{thm:oracle-perf-bound:e5} in \eqref{thm:oracle-perf-bound:e6:s1}. Setting the tunable weight $\eta$ in \eqref{thm:oracle-perf-bound:e6:s2} to be
\[\eta = \left(\frac{\frac{\ell}{2} \cdot \left(1 + 2 C_3 L_g^2\right)\cdot \left(1 + C_3\right)\cdot \sum_{t=0}^{T-1} e_t^2}{\cost(\OPT)}\right)^{\frac{1}{2}}\]
gives that
\begin{align}\label{thm:oracle-perf-bound:final-conclusion}
    &\cost(\ALG) - \cost(\OPT)\nonumber\\
    \leq{}& \sqrt{\left(\frac{\ell}{2} \cdot \left(1 + 2 C_3 L_g^2\right)\cdot \left(1 + C_3\right)\right) \cdot \cost(\OPT) \cdot \sum_{t=0}^{T-1} e_t^2}\nonumber\\
    &+ \frac{\ell}{2} \cdot \left(1 + 2 C_3 L_g^2\right)\cdot \left(1 + C_3\right)\cdot \sum_{t=0}^{T-1} e_t^2.
\end{align}
This finishes the proof of \Cref{thm:per-step-error-to-performance-guarantee}.

\section{Proof of Theorem \ref{thm:the-pipeline-theorem}}\label{appendix:the-pipeline-theorem}
We first use induction to show that the following two conditions holds for all time steps $t < T$:
\begin{subequations}\label{thm:the-pipeline-theorem:e1}
\begin{align}
    x_t &\in \mathcal{B}\left(x_t^*, \frac{R}{C_3}\right),\label{thm:the-pipeline-theorem:e1:s1}\\
    e_t &\leq \sum_{\tau = 0}^{k}\left(\left(\frac{R}{C_3} + D_{x^*}\right) \cdot q_1(\tau) + q_2(\tau)\right)\rho_{t, \tau} + 2R\left(\left(\frac{R}{C_3} + D_{x^*}\right) \cdot q_1(k) + q_2(k)\right). \label{thm:the-pipeline-theorem:e1:s2}
\end{align}
\end{subequations}
At time step $0$, \eqref{thm:the-pipeline-theorem:e1:s1} holds because $x_0 = x_0^*$, and \eqref{thm:the-pipeline-theorem:e1:s2} holds by \Cref{lemma:pipeline-step2} and the assumption on the terminal cost $F_k$ of $\MPC_k$.

Suppose \eqref{thm:the-pipeline-theorem:e1:s1} and \eqref{thm:the-pipeline-theorem:e1:s2} hold for all time steps $\tau < t$. For time step $t$, by the assumption on the prediction errors $\rho_{t, \tau}$ and prediction horizon $k$ in \Cref{thm:the-pipeline-theorem}, we know that $e_{\tau} \leq \frac{R}{C_3^2 L_g}$ holds for all $\tau < t$ because \eqref{thm:the-pipeline-theorem:e1:s2} holds for all $\tau < t$. Thus, we know that \eqref{thm:the-pipeline-theorem:e1:s1} holds for time step $t$ by \Cref{thm:per-step-error-to-performance-guarantee}. Then, since \eqref{thm:the-pipeline-theorem:e1:s1} holds for time step $t$, and the terminal cost $F_{t+k}$ of $\MPC_k$ is set to be the indicator function of some state $\bar{y}(\xi_{t+k\mid t})$ that satisfies $\bar{y}(\xi_{t+k\mid t}) \in \mathcal{B}(x_{t+k}^*, R)$ if $t < T-k$, we know \eqref{thm:the-pipeline-theorem:e1:s2} also holds for time step $t$ by \Cref{lemma:pipeline-step2}. This finishes the induction proof of \eqref{thm:the-pipeline-theorem:e1}.

To simplify the notation, let $R_0 \coloneqq \frac{R}{C_3} + D_{x^*}$. Note that \eqref{thm:the-pipeline-theorem:e1:s2} implies that
\begin{subequations}\label{thm:the-pipeline-theorem:e2}
\begin{align}
    e_t^2 \leq{}& \left(\sum_{\tau = 0}^{k}\left(R_0 \cdot q_1(\tau) + q_2(\tau)\right) + 2R\left(R_0 + 1\right)\right)\nonumber\\
    &\cdot \left(\sum_{\tau = 0}^{k}\left(R_0 \cdot q_1(\tau) + q_2(\tau)\right)\rho_{t, \tau}^2 + 2R\left(R_0 \cdot q_1(k)^2 + q_2(k)^2\right)\right) \label{thm:the-pipeline-theorem:e2:s1}\\
    \leq{}& \left(R_0 C_1 + C_2 + 2R(R_0 + 1)\right)\nonumber\\
    &\cdot \left(\sum_{\tau = 0}^{k-1}\left(R_0 \cdot q_1(\tau) + q_2(\tau)\right)\rho_{t, \tau}^2 + (2R+1)\left(R_0 \cdot q_1(k)^2 + q_2(k)^2\right)\right), \label{thm:the-pipeline-theorem:e2:s2}
\end{align}
\end{subequations}
where we use the Cauchy-Schwarz inequality in \eqref{thm:the-pipeline-theorem:e2:s1}; we use the bounds $\sum_{\tau = 0}^k q_1(\tau) \leq C_1$, $\sum_{\tau = 0}^k q_2(\tau) \leq C_2$, and $\rho_{t, \tau} \leq 1$ in \eqref{thm:the-pipeline-theorem:e2:s2}.

Since \eqref{thm:the-pipeline-theorem:e1} and \eqref{thm:the-pipeline-theorem:e2} holds for all time steps $t < T$, we can apply \Cref{thm:per-step-error-to-performance-guarantee} to obtain that
\begin{equation*}
    \cost(\MPC_k) - \cost(\OPT) \leq \sqrt{\cost(\OPT) \cdot E_0} + E_0,
\end{equation*}
where
\begin{align*}
    E_0 \coloneqq{}& \left(R_0 C_1 + C_2 + 2R(R_0 + 1)\right)\nonumber\\
    &\cdot \left(\sum_{\tau = 0}^{k-1}\left(R_0 \cdot q_1(\tau) + q_2(\tau)\right)P(\tau) + (2R+1)\left(R_0 \cdot q_1(k)^2 + q_2(k)^2\right)T\right).
\end{align*}
This finishes the proof of \Cref{thm:the-pipeline-theorem}.

\section{Assumptions and Proofs of Section \ref{sec:unconstrained:disturbances}}\label{appendix:unconstrained-LTV-disturbances}
The formal definition of the controllability index $d$ and $\sigma$-uniform controllable are given in \cite{lin2021perturbation}. For completeness, we restate them for LTV dynamics in Assumption \ref{assump:unconstrained-LTV-pred-err-disturbance} below.

\begin{assumption}\label{assump:unconstrained-LTV-pred-err-disturbance}
For time steps $t_2 \geq t_1$, we define the transition matrix $\Phi(t_2, t_1) \in \mathbb{R}^{n \times n}$ as
    \[\Phi(t_2, t_1) \coloneqq \begin{cases}
        A_{t_2 - 1} A_{t_2 - 2} \cdots A_{t_1} & \text{ if }t_2 > t_1\\
        I & \text{ otherwise.}
    \end{cases},\]
For any positive integer $p$, we define the controllability matrix $M(t, p) \in \mathbb{R}^{n \times (mp)}$ as
    \[M(t, p) \coloneqq \left[\Phi(t+p, t+1) B_t, \Phi(t+p, t+2) B_{t+1}, \ldots, \Phi(t+p, t+p) B_{t+p}\right].\]
We assume the LTV system $\{A_t, B_t\}$ is $\sigma$-uniform controllable with controllability index $d$, i.e., $d$ is the smallest positive integer such that $\sigma_{min}\left(M(t, d)\right) > 0$ holds for all $t \in [0, T - d]$, and $\sigma_{min}\left(M(t, d)\right) \geq \sigma$ holds for all $t \in [0, T - d]$.
\end{assumption}

As a remark, the Assumption \ref{assump:unconstrained-LTV-pred-err-disturbance} is a special case of \Cref{def:controllability} in unconstrained LTV systems. \cite{lin2021perturbation} has established a perturbation bound for the LTV system in \eqref{equ:online_control_problem:unconstrained-LTV-disturbance} which implies the our requirements in Property \ref{assump:pipeline-perturbation-bounds}. Thus we can use \Cref{thm:the-pipeline-theorem} to show \Cref{thm:perturbation:unconstrained-LTV-pred-err-disturbance}.

\begin{proof}[Proof of \Cref{thm:perturbation:unconstrained-LTV-pred-err-disturbance}]
By Theorem 3.3 in \cite{lin2021perturbation}, we know Property \ref{assump:pipeline-perturbation-bounds} holds under Assumption \ref{assump:unconstrained:disturbances} for arbitrary $R$ and $q_1(t) = 0, q_2(t) = H_1 \lambda_1^t$, and $q_3(t) = H_1 \lambda_1^t$, where $H_1 = H_1(\mu, \ell, d, \sigma, a, b, b', L_w) > 0$ is some constant, and $\lambda_1 = \lambda_1(\mu, \ell, d, \sigma, a, b, b', L_w) \in (0, 1)$ is the decay rate. Here, $H_1$ corresponds to $C$ and $\lambda_1$ corresponds to $\lambda$ in Theorem 3.3 in \cite{lin2021perturbation}.

By setting $R \coloneqq \max\left\{ D_{x^*}, \frac{2L_g H_1^3}{(1 - \lambda_1)^3} \right\}$, we guarantee that the terminal state $0$ of $\MPC_k$ is always in the closed ball $\mathcal{B}(x_{t+k}^*, R)$, and the condition
\[\sum_{\tau = 0}^{k}\left(\left(\frac{R}{C_3} + D_{x^*}\right) \cdot q_1(\tau) + q_2(\tau)\right)\rho_{t, \tau} + 2R\left(\left(\frac{R}{C_3} + D_{x^*}\right) \cdot q_1(k) + q_2(k)\right) \leq \frac{R}{C_3^2 L_g}\]
holds once $k \geq \ln\left(\frac{4 H_1^3 L_g}{(1 - \lambda_1)^2}\right)/\ln(1/\lambda_1)$ because $\rho_{t, \tau} \leq 1$. Therefore, we can apply \Cref{thm:the-pipeline-theorem} to finish the proof of \Cref{thm:perturbation:unconstrained-LTV-pred-err-disturbance}.
\end{proof}

\section{Assumptions and Proofs of Section \ref{sec:unconstrained:dynamics}}\label{appendix:unconstrained-LTV-dynamics}
In this section, we give the detailed assumptions and proofs of the results in Section \ref{sec:unconstrained:dynamics}. Before we present the assumption on the uncertain LTV systems in \eqref{equ:online_control_problem:unconstrained-LTV-dynamics}, we first define several quantities that we will use heavily in the rest of this section:

For time steps $t_1 \leq t_2$ and $\xi_{t_1:t_2} \in \varXi_{t_1:t_2}$, define
\begin{align}\label{equ:large-control-matrix-with-terminal}
    N_{t_1}^{t_2}(\xi_{t_1:t_2}) &\coloneqq \begin{bmatrix}
    I & & & & \\
    - A_{t_1}(\xi_{t_1}) & - B_{t_1}(\xi_{t_1}) & I & & \\
     & & \ddots & & & & \\
     & & - A_{t_2}(\xi_{t_2}) & - B_{t_2}(\xi_{t_2}) & I
    \end{bmatrix}.
\end{align}
This matrix is closely related to the stability of the LTV system in the time interval $[t_1, t_2 + 1]$. To see this, note that $N_{t_1}^{t_2}(\xi_{t_1:t_2})$ always has full row rank, i.e., given any disturbance vector $w = \left(x_{t_1}, w_{t_1}, w_{t_1 + 1}, \ldots, w_{t_2}\right)^\top$, one can always find a feasible sub-trajectory $z = \left(x_{t_1}, u_{t_1}, x_{t_1 + 1}, \ldots, u_{t_2}, x_{t_2 + 1}\right)^\top$ that satisfies $N_{t_1}^{t_2}(\xi_{t_1:t_2}) z = w$. If for any vector $w$, there exists a feasible sub-trajectory $z$ such that $\norm{z} \leq (1/\sigma) \cdot \norm{w}$ for some positive constant $\sigma$, then the smallest singular value of $N_{t_1}^{t_2}(\xi_{t_1:t_2})$ is lower bounded by $\sigma$.

Similar with \eqref{equ:large-control-matrix-with-terminal}, we define matrix
\begin{align}\label{equ:large-control-matrix-without-terminal}
    \hat{N}_{t_1}^{t_2}(\xi_{t_1:t_2}) &\coloneqq \begin{bmatrix}
    I & & & \\
    - A_{t_1}(\xi_{t_1}) & - B_{t_1}(\xi_{t_1}) & I & \\
     & & \ddots & & & \\
     & & - A_{t_2}(\xi_{t_2}) & - B_{t_2}(\xi_{t_2})
    \end{bmatrix}
\end{align}
for any time steps $t_1 \leq t_2$ and $\xi_{t_1:t_2} \in \varXi_{t_1:t_2}$, which removes the last column of \eqref{equ:large-control-matrix-with-terminal}. The matrix $\hat{N}_{t_1}^{t_2}(\xi_{t_1:t_2})$ is closely related to the controllability of the LTV system in the time interval $[t_1, t_2 + 1]$. To see this, given any disturbance vector $\hat{w} = \left(x_{t_1}, w_{t_1}, w_{t_1 + 1}, \ldots, w_{t_2} - x_{t_2+1}\right)^\top$ whose first/last entry depends on the initial/terminal state, a feasible sub-trajectory $\hat{z} = \left(x_{t_1}, u_{t_1}, x_{t_1 + 1}, \ldots, u_{t_2}\right)^\top$ must satisfy that $\hat{N}_{t_1}^{t_2}(\xi_{t_1:t_2}) \hat{z} = \hat{w}$. Different from $N_{t_1}^{t_2}(\xi_{t_1:t_2})$, $\hat{N}_{t_1}^{t_2}(\xi_{t_1:t_2})$ is not guaranteed to have full row rank. If for any vector $\hat{w}$, there exists a feasible sub-trajectory $\hat{z}$ such that $\norm{\hat{z}} \leq (1/\sigma)\cdot \norm{\hat{w}}$ for some positive constant $\sigma$, then the smallest singular value of $\hat{N}_{t_1}^{t_2}(\xi_{t_1:t_2})$ is lower bounded by $\sigma$.

We make the following assumption on the smallest singular values of matrices $N_{t_1}^{t_2}(\xi_{t_1:t_2})$ and $\hat{N}_{t_1}^{t_2}(\xi_{t_1:t_2})$ so that the LTV system possesses uniform stability and controllability properties under any uncertainty parameters:

\begin{assumption}\label{assump:unconstrained-LTV-pred-err-dynamics}
There exists some universal constant $\sigma > 0$ such that $\sigma_{min}\left(N_{t}^{T-1}(\xi_{t:T-1})\right) \geq \sigma$ for any $t < T$, and $\sigma_{min}\left(\hat{N}_{t}^{t + k}\left(\xi_{t:t+k}\right)\right) \geq \sigma$ for any $t < T - k$.
\end{assumption}

While Assumption \ref{assump:unconstrained-LTV-pred-err-dynamics} may seem more restricted than the uniform controllability defined in \Cref{def:controllability}, it can actually be derived from \Cref{def:controllability} by Lemma 12 in \cite{shin2021controllability}.

In order to formulate \eqref{equ:online_control_problem:unconstrained-LTV-dynamics} as a quadratic programming problem with equality constraints, we also need to define the matrix for cost functions:
\begin{align}\label{equ:large-control-matrix-costs}
    M_{t}^{T}(\xi_{t:T}) &\coloneqq \diag\left(Q_{t}(\xi_{t}), R_{t}(\xi_{t}), Q_{t+1}(\xi_{t+1}), \ldots, R_{T-1}(\xi_{T-1}), P_T(\xi_T)\right), \forall t < T,\nonumber\\
    \hat{M}_{t}^{t+k}(\xi_{t:t+k}) &\coloneqq \diag\left(Q_{t}(\xi_{t}), R_{t}(\xi_{t}), Q_{t+1}(\xi_{t+1}), \ldots, R_{t+k-1}(\xi_{t+k-1})\right), \forall t < T - k.
\end{align}
\begin{comment}
With some slight abuse of the notations, we will use
\begin{align*}
    &N(\xi_{t:T}) \coloneqq N_{t}^{T}(\xi_{t:T}), M(\xi_{t:T}) \coloneqq M_{t}^{T}(\xi_{t:T}), \forall t < T\\
    &\hat{N}(\xi_{t:t+k}) \coloneqq \hat{N}_{t}^{t+k}(\xi_{t:t+k}), \hat{M}(\xi_{t:t+k}) \coloneqq
\end{align*}

$N(\xi_{t:T}) \coloneqq N_{t}^{T}(\xi_{t:T})$, $\hat{N}(\xi_{t:T}) \coloneqq \hat{N}_{t}^{T}(\xi_{t:t+k})$, $M(\xi_)$ to simplify the notations.
\end{comment}
To write down the KKT condition of the equality constrained quadratic programming problem, we also need to define
\begin{align*}
    H_t^T(\xi_{t:T}) \coloneqq{}& \begin{bmatrix}
    M_t^T(\xi_{t:T}) & N_t^{T-1}(\xi_{t:T-1})^{\top}\\
    N_t^{T-1}(\xi_{t:T-1}) & 0
    \end{bmatrix},\\
    \hat{H}_t^{t+k}(\xi_{t:t+k}) \coloneqq{}& \begin{bmatrix}
    \hat{M}_t^{t+k}(\xi_{t:t+k}) & \hat{N}_t^{t+k-1}(\xi_{t:t+k-1})^{\top}\\
    \hat{N}_t^{t+k-1}(\xi_{t:t+k-1}) & 0
    \end{bmatrix}, \\
    b_t^T(z, \xi_{t:T}) \coloneqq{}& \left(Q_{t}(\xi_{t})\bar{x}_{t}(\xi_{t}), 0,  \ldots, P(\xi_{T})\bar{x}_{T}(\xi_{T}), z, w_{t}(\xi_{t}), \ldots, w_{T-1}(\xi_{T-1})\right)^{\top},\\
    \hat{b}_t^{t+k}(z, \xi_{t:t+k}) \coloneqq{}& \left(Q_{t}(\xi_{t})\bar{x}_{t}(\xi_{t}), 0,  \ldots, 0, z, w_{t}(\xi_{t}), \ldots, w_{t+k-1}(\xi_{t+k-1}) - \xi_{t+k}\right)^{\top},\\
    \chi_t^T ={}& \left(y_{t}, v_{t}, y_{t+1}, \ldots, v_{T-1}, y_{T}, \eta_{t}, \eta_{t+1}, \ldots, \eta_{T}\right)^{\top},\\
    \hat{\chi}_t^{t+k} ={}& \left(y_{t}, v_{t}, y_{t+1}, \ldots, v_{t+k-1}, \eta_{t}, \eta_{t+1}, \ldots, \eta_{t+k}\right)^{\top}.
\end{align*}
According to the KKT condition, the optimal primal-dual solution to $\iota_t^T(z, \xi_{t:T}; F_T)$ ($t < T$) is the unique solution $\chi_t^T$ to the linear equation $H_t^T(\xi_{t:T}) \chi_t^T = b_t^T(z, \xi_{t:T})$. Similarly, the optimal primal-dual solution to $\iota_t^{t+k}(z, \xi_{t:t+k}; \mathbb{I})$ ($t < T - k$) is the unique solution $\chi_t^{t+k}$ to the linear equation $\hat{H}_t^{t+k}(\xi_{t:t+k}) \hat{\chi}_t^{t+k} = \hat{b}_t^{t+k}(z, \xi_{t:t+k})$. We provide an illustrative example for $\chi_t^T$ with $(t, T) = (0, 3)$ below:
\setcounter{MaxMatrixCols}{20}
{\small
\begin{align*}
    \left[\begin{array}{ccccccc|cccc}
    Q_0 & & & & & & & I & - A_0^{\top} & & \\
     & R_0 & & & & & & & - B_0^{\top} & & \\
     & & Q_1 & & & & & & I & - A_1^{\top} & \\
     & & & R_1 & & & & & & - B_1^{\top} & \\
     & & & & Q_2 & & & & & I & - A_2^{\top} \\
     & & & & & R_2 & & & & & - B_2^{\top} \\
     & & & & & & P_3 & & & & I \\
     \hline
    I & & & & & & & & & & \\
    - A_0 & - B_0 & I & & & & & & & & \\
     & & -A_1 & -B_1 & I & & & & & & \\
     & & & & -A_2 & -B_2 & I & & & &
    \end{array}\right] \left[\begin{array}{c}
    y_0 \\
    v_0 \\
    y_1 \\
    v_1 \\
    y_2 \\
    v_2 \\
    y_3 \\
    \hline
    \eta_0 \\
    \eta_1 \\
    \eta_2 \\
    \eta_3
    \end{array}\right] = \left[\begin{array}{c}
    Q_0 \bar{x}_0 \\
    0 \\
    Q_1 \bar{x}_1 \\
    0 \\
    Q_2 \bar{x}_2 \\
    0 \\
    P_3 \bar{x}_3 \\
    \hline
    z \\
    w_0 \\
    w_1 \\
    w_2
    \end{array}\right],
\end{align*}
}where we omit the parameters $\xi_{0:3}$ to simplify the notations. Rearranging the rows and columns of the matrix on the left hand side gives the equation:
{\small
\begin{align*}
    \left[\begin{array}{ccc|ccc|ccc|cc}
        Q_0 & & I & & & - A_0^{\top} & & & & & \\
         & R_0 & & & & - B_0^{\top} & & & & & \\
        I & & & & & & & & & & \\
        \hline
         & & & Q_1 & & I & & & -A_1^{\top} & & \\
         & & & & R_1 & & & & - B_1^{\top} & & \\
        -A_0 & - B_0 & & I & & & & & & & \\
        \hline
         & & & & & & Q_2 & & I & & -A_2^{\top} \\
         & & & & & & & R_2 & & & -B_2^{\top} \\
         & & & -A_1 & -B_1 & & I & & & & \\
        \hline
         & & & & & & & & & P & I\\
         & & & & & & -A_2 & -B_2 & & I &
    \end{array}\right] \left[\begin{array}{c}
    y_0 \\
    v_0 \\
    \eta_0 \\
    \hline
    y_1 \\
    v_1 \\
    \eta_1 \\
    \hline
    y_2 \\
    v_2 \\
    \eta_2\\
    \hline
    y_3 \\
    \eta_3
    \end{array}\right] = \left[\begin{array}{c}
    Q_0 \bar{x}_0 \\
    0 \\
    z \\
    \hline
    Q_1 \bar{x}_1 \\
    0 \\
    w_0\\
    \hline
    Q_2 \bar{x}_2 \\
    0 \\
    w_1 \\
    \hline
    P \bar{x}_3 \\
    w_2
    \end{array}\right].
\end{align*}
}Let $\Phi_t^T$ denote the permutation matrix that permute $(y_t, v_t, y_{t+1}, \ldots, v_{T-1}, y_T; \eta_t, \ldots, \eta_T)^\top$ to $(y_t, v_t, \eta_t, y_{t+1}, v_{t+1}, \eta_{t+1}, \ldots, y_T, \eta_T)^\top$.
We use $\Upsilon_t^T(\xi_{t:T}) \coloneqq (\Phi_t^T) H_t^T(\xi_{t:T}) (\Phi_t^T)^\top$ to denote the rearrangement of $H_t^T(\xi_{t:T})$ as illustrated in the above equation, and use $\beta_t^T(z, \xi_{t:T}) \coloneqq (\Phi_t^T) b_t^T(z, \xi_{t:T})$ to denote the corresponding rearrangement of $b_t^T(z, \xi_{t:T})$.

We also provide an illustrative example for $\hat{\chi}_t^{t+k}$ with $(t, k) = (0, 3)$ below:
{\small
\begin{align*}
    \left[\begin{array}{cccccc|cccc}
    Q_0 & & & & & & I & - A_0^{\top} & & \\
     & R_0 & & & & & & - B_0^{\top} & & \\
     & & Q_1 & & & & & I & - A_1^{\top} & \\
     & & & R_1 & & & & & - B_1^{\top} & \\
     & & & & Q_2 & & & & I & - A_2^{\top} \\
     & & & & & R_2 & & & & - B_2^{\top} \\
     \hline
    I & & & & & & & & & \\
    - A_0 & - B_0 & I & & & & & & & \\
     & & -A_1 & -B_1 & I & & & & & \\
     & & & & -A_2 & -B_2 & & & &
    \end{array}\right] \left[\begin{array}{c}
    y_0 \\
    v_0 \\
    y_1 \\
    v_1 \\
    y_2 \\
    v_2 \\
    \hline
    \eta_0 \\
    \eta_1 \\
    \eta_2 \\
    \eta_3
    \end{array}\right] = \left[\begin{array}{c}
    Q_0 \bar{x}_0 \\
    0 \\
    Q_1 \bar{x}_1 \\
    0 \\
    Q_2 \bar{x}_2 \\
    0 \\
    \hline
    z \\
    w_0 \\
    w_1 \\
    w_2
    \end{array}\right],
\end{align*}
}where we omit the parameters $\xi_{0:3}$ to simplify the notations. Rearranging the rows and columns of the matrix on the left hand side gives the equation:
{\small
\begin{align*}
    \left[\begin{array}{ccc|ccc|ccc|c}
        Q_0 & & I & & & - A_0^{\top} & & & &  \\
         & R_0 & & & & - B_0^{\top} & & & &  \\
        I & & & & & & & & &  \\
        \hline
         & & & Q_1 & & I & & & -A_1^{\top} &  \\
         & & & & R_1 & & & & - B_1^{\top} &  \\
        -A_0 & - B_0 & & I & & & & & &  \\
        \hline
         & & & & & & Q_2 & & I &  -A_2^{\top} \\
         & & & & & & & R_2 & &  -B_2^{\top} \\
         & & & -A_1 & -B_1 & & I & & &  \\
        \hline
         & & & & & & -A_2 & -B_2 & &
    \end{array}\right] \left[\begin{array}{c}
    y_0 \\
    v_0 \\
    \eta_0 \\
    \hline
    y_1 \\
    v_1 \\
    \eta_1 \\
    \hline
    y_2 \\
    v_2 \\
    \eta_2\\
    \hline
    \eta_3
    \end{array}\right] = \left[\begin{array}{c}
    Q_0 \bar{x}_0 \\
    0 \\
    z_0 \\
    \hline
    Q_1 \bar{x}_1 \\
    0 \\
    w_0\\
    \hline
    Q_2 \bar{x}_2 \\
    0 \\
    w_1 \\
    \hline
    w_2 - z_3
    \end{array}\right].
\end{align*}
}Let $\hat{\Phi}_t^{t+k}$ denote the permutation matrix that permute $(y_t, v_t, y_{t+1}, \ldots, v_{t+k-1}; \eta_t, \ldots, \eta_{t+k})^\top$ to $(y_t, v_t, \eta_t, y_{t+1}, v_{t+1}, \eta_{t+1}, \ldots, y_{t+k-1}, v_{t+k-1}, \eta_{t+k-1}, \eta_{t+k})^\top$. We use $\hat{\Upsilon}_t^{t+k}(\xi_{t:t+k})\coloneqq (\hat{\Phi}_t^{t+k}) \hat{H}_t^{t+k}(\xi_{t:t+k}) (\hat{\Phi}_t^{t+k})^\top$ to denote the rearrangement of $\hat{H}_t^{t+k}(\xi_{t:t+k})$ as illustrated in the above equation, and use $\hat{\beta}_t^{t+k}(z, \xi_{t:t+k}) \coloneqq (\hat{\Phi}_t^{t+k}) \hat{b}_t^{t+k}(z, \xi_{t:t+k})$ to denote the corresponding rearrangement of $\hat{b}_t^{t+k}(z, \xi_{t:t+k})$.

Before showing the main result about the per-step error, we first show a technical lemma about the singular values of a block matrix in \Cref{lemma:block-singular-value}.

\begin{lemma}\label{lemma:block-singular-value}
Consider a block matrix
\begin{align*}
    H = \begin{bmatrix}
    M & N^{\top}\\
    N & 0
    \end{bmatrix}.
\end{align*}
Here $M \in \mathbb{R}^{n_0 \times n_0}$ is a symmetric positive definite matrix that satisfies $\underline{\sigma}_M I \preceq M \leq \overline{\sigma}_M I$ with $\underline{\sigma}_M > 0$, and $N \in \mathbb{R}^{n_1 \times n_0}$ with $n_1 \leq n_0$ satisfies that $\underline{\sigma}_N \leq \sigma(N) \leq \overline{\sigma}_N$ with $\underline{\sigma}_N > 0$. Then $H$ satisfies that
\[\min(\underline{\sigma}_M, 1) \cdot \overline{\sigma}_N \cdot \sqrt{\frac{\overline{\sigma}_M}{2\underline{\sigma}_M \overline{\sigma}_M + \underline{\sigma}_M (\underline{\sigma}_N)^2}} \leq \sigma(H) \leq \sqrt{2} (\overline{\sigma}_M + \overline{\sigma}_N).\]
\end{lemma}
\begin{proof}[Proof of Lemma \ref{lemma:block-singular-value}]
We first establish the lower bound on $\sigma(H)$: Suppose the singular value decomposition of $N M^{-\frac{1}{2}}$ is given by
\[N M^{-\frac{1}{2}} = U \Sigma V^{\top},\]
where $U \in \mathbb{R}^{n_1 \times n_1}$ and $V \in \mathbb{R}^{n_0 \times n_0}$ are unitary matrices and $\Sigma = diag(\sigma_1, \sigma_2, \ldots, \sigma_{n_1}) \in \mathbb{R}^{n_1 \times n_0}$ with $\frac{\overline{\sigma}_N}{\sqrt{\underline{\sigma}_M}} \geq \sigma_1 \geq \cdots \geq \sigma_m \geq \frac{\underline{\sigma}_N}{\sqrt{\overline{\sigma}_M}}$. We can decompose $H$ as
\begin{align}\label{lemma:block-singular-value:e1}
    H ={}& \begin{bmatrix}
    M^{\frac{1}{2}} & 0\\
    0 & I
    \end{bmatrix}\cdot \begin{bmatrix}
    I & M^{-\frac{1}{2}} F^{\top}\\
    F M^{-\frac{1}{2}} & 0
    \end{bmatrix}\cdot \begin{bmatrix}
    M^{\frac{1}{2}} & 0\\
    0 & I
    \end{bmatrix}\nonumber\\
    ={}& \begin{bmatrix}
    M^{\frac{1}{2}} & 0\\
    0 & I
    \end{bmatrix}\cdot \begin{bmatrix}
    V & 0\\
    0 & U
    \end{bmatrix}\cdot \begin{bmatrix}
    I & \Sigma^{\top}\\
    \Sigma & 0
    \end{bmatrix}\cdot \begin{bmatrix}
    V^{\top} & 0\\
    0 & U^{\top}
    \end{bmatrix} \cdot \begin{bmatrix}
    M^{\frac{1}{2}} & 0\\
    0 & I
    \end{bmatrix}.
\end{align}
Note that for any $\alpha \in \mathbb{R}^{n_0}$ and $\beta \in \mathbb{R}^{n_1}$, we have
\begin{align*}
    \norm{\begin{bmatrix}
    I & \Sigma^{\top}\\
    \Sigma & 0
    \end{bmatrix} \begin{bmatrix}
    \alpha\\
    \beta
    \end{bmatrix}}^2 ={}& \sum_{i=1}^{n_1} (\alpha_i + \sigma_i \beta_i)^2 + \sum_{i=n_1+1}^{n_0} \alpha_i^2 + \sum_{i=1}^{n_1} \sigma_i^2 \alpha_i^2\\
    ={}& \sum_{i=1}^{n_1} \left(\left(1 + \frac{\sigma_i^2}{2}\right)\left(\alpha_i +\frac{2\sigma_i}{2 + \sigma_i^2} \beta_i\right)^2 + \frac{\sigma_i^2}{2}\alpha_i^2 + \frac{\sigma_i^2}{2 + \sigma_i^2}\beta_i^2\right) + \sum_{i=n_1+1}^{n_0} \alpha_i^2\\
    \geq{}& \left(\min_i \frac{\sigma_i^2}{2 + \sigma_i^2}\right) \norm{\begin{bmatrix}
    \alpha\\
    \beta
    \end{bmatrix}}^2\\
    \geq{}& \frac{\underline{\sigma}_M (\underline{\sigma}_N)^2}{2\underline{\sigma}_M \overline{\sigma}_M + \overline{\sigma}_M (\overline{\sigma}_N)^2} \norm{\begin{bmatrix}
    \alpha\\
    \beta
    \end{bmatrix}}^2.
\end{align*}
Therefore, by \eqref{lemma:block-singular-value:e1}, we see that
\[\norm{H \begin{bmatrix}
    \alpha\\
    \beta
    \end{bmatrix}} \geq \min(\underline{\sigma}_M, 1) \cdot \underline{\sigma}_N \cdot \sqrt{\frac{\underline{\sigma}_M}{2\underline{\sigma}_M \overline{\sigma}_M + \overline{\sigma}_M (\overline{\sigma}_N)^2}} \cdot \norm{\begin{bmatrix}
    \alpha\\
    \beta
    \end{bmatrix}}.\]
This finishes the proof of the lower bound.

For the upper bound, note that
\begin{align*}
    \norm{H \begin{bmatrix}
    \alpha\\
    \beta
    \end{bmatrix}} &\leq \norm{M \alpha + N^{\top} \beta} + \norm{N \alpha}\\
    &\leq  \norm{M \alpha} + \norm{N^{\top} \beta} + \norm{N \alpha}\\
    &\leq (\overline{\sigma}_M + \overline{\sigma}_N)(\norm{x} + \norm{y})\\ &\leq \sqrt{2} (\overline{\sigma}_M + \overline{\sigma}_N) \norm{\begin{bmatrix}
    \alpha\\
    \beta
    \end{bmatrix}}.
\end{align*}
\end{proof}

%Let $\Upsilon_{t_1}^{t_2}(\xi_{t_1:t_2})$ denote the tri-diagonal rearrangement of $H_{t_1}^{t_2}(\xi_{t_1:t_2})$ as the above example. In general, $\Upsilon_{t_1}^{t_2}(\xi_{t_1:t_2})$ is a $(t_2 - t_1 + 1) \times (t_2 - t_1 + 1)$ block matrix, indexed by $(i, j) \in [t_1, t_2]^2$. Similarly, let $\hat{\Upsilon}_{t_1}^{t_2}(\xi_{t_1:t_2})$ denote the tri-diagonal rearrangement of $\hat{H}_{t_1}^{t_2}(\xi_{t_1:t_2})$. $\Upsilon_{t_1}^{t_2}(\xi_{t_1:t_2})$ is also a $(t_2 - t_1 + 1) \times (t_2 - t_1 + 1)$ block matrix like $\Upsilon_{t_1}^{t_2}(\xi_{t_1:t_2})$, although the dimension of the block on the last row and column is different.

Since the primal-dual optimal solution to $\iota_t^T(z, \xi_{t:T}; F_T)$ and $\iota_t^{t+k}(z, \xi_{t:t+k}; \mathbb{I})$ are given by $\left(\Upsilon_t^T(\xi_{t:T})\right)^{-1} \beta_t^T(z, \xi_{t:T})$ and $\left(\hat{\Upsilon}_t^{t+k}(\xi_{t:t+k})\right)^{-1} \hat{\beta}_t^{t+k}(z, \xi_{t:t+k})$ respectively, it is critical to establish the exponentially decaying bounds for the matrices $\left(\Upsilon_t^T(\xi_{t:T})\right)^{-1}$ and $\left(\hat{\Upsilon}_t^{t+k}(\xi_{t:t+k})\right)^{-1}$. Note that after the rearrangement, $\Upsilon_t^T(\xi_{t:T})$ is a block matrix with $(T - t + 1) \times (T - t + 1)$ blocks, indexed by $(i, j) \in [t, T]^2$; $\Upsilon_t^{t+k}(\xi_{t:t+k})$ is a block matrix with $(k + 1) \times (k + 1)$ blocks, indexed by $(i, j) \in [t, t+k]^2$.

\begin{lemma}\label{thm:diff-inverse-sophisticated}
Under Assumption \ref{assump:unconstrained:dynamics}, the following inequalities hold for the norm of the block entries of $\left(\Upsilon_t^T(\xi_{t:T})\right)^{-1}$ and $\left(\hat{\Upsilon}_t^{t+k}(\xi_{t:t+k})\right)^{-1}$:
\begin{align}\label{thm:diff-inverse-sophisticated:e1}
    \norm{\left(\Upsilon_{t}^{T}(\xi_{t:T})^{-1}\right)_{i j}} &\leq C_2 \lambda_2^{\abs{i - j}}, \forall (i, j) \in [t, T]^2, \forall \xi_{t:T} \in \varXi_{t:T},\nonumber\\
    \norm{\left(\hat{\Upsilon}_{t}^{t+k}(\xi_{t:t+k})^{-1}\right)_{i j}} &\leq C_2 \lambda_2^{\abs{i - j}}, \forall (i, j) \in [t, t+k]^2, \forall \xi_{t:t+k} \in \varXi_{t:t+k}.
\end{align}
Further, the following inequalities hold for the norm of differences between the block entries of $\left(\Upsilon_t^T(\xi_{t:T})\right)^{-1}$ and $\left(\hat{\Upsilon}_t^{t+k}(\xi_{t:t+k})\right)^{-1}$: For all $(i, j) \in [t, T]^2$ and $\xi_{t:T} \in \varXi_{t:T}$, we have
\begin{align}\label{thm:diff-inverse-sophisticated:e2}
    \norm{\left(\Upsilon_{t}^{T}(\xi_{t:T})^{-1} - \Upsilon_{t}^{T}(\xi_{t:T}')^{-1}\right)_{i j}} \leq C_2' \sum_{\tau=t}^{T} \lambda_2^{\abs{\tau - i} + \abs{\tau - j}}\cdot \norm{\xi_\tau - \xi_\tau'},
\end{align}
For all $(i, j) \in [t, t+k]^2$ and $\xi_{t:T} \in \varXi_{t:t+k}$ with $t < T - k$, we have
\begin{align}\label{thm:diff-inverse-sophisticated:e2-1}
    \norm{\left(\hat{\Upsilon}_{t}^{t+k}(\xi_{t:t+k})^{-1} - \hat{\Upsilon}_{t}^{t+k}(\xi_{t:t+k}')^{-1}\right)_{i j}} \leq C_2' \sum_{\tau=t}^{t+k} \lambda_2^{\abs{\tau - i} + \abs{\tau - j}}\cdot \norm{\xi_\tau - \xi_\tau'},
\end{align}
where the constants $C_2, C_2',$ and $\lambda_2$ are given by
\begin{align*}
    \lambda_2 &= \left(\frac{\overline{\sigma}_H - \underline{\sigma}_H}{\overline{\sigma}_H + \underline{\sigma}_H}\right)^{\frac{1}{2}}, C_2 = \frac{4(\ell + 1 + a + b)}{\underline{\sigma}_H^2 \cdot \lambda_2},\\
    C_2' &= C_2^2 \left(\max\{L_Q + L_R, L_P\} + \frac{2}{\lambda_2}(L_A + L_B)\right).
\end{align*}
where $\underline{\sigma}_H$ and $\overline{\sigma}_H$ are defined as
\[\underline{\sigma}_H \coloneqq \min(\mu, 1) \cdot (a + b + 1) \cdot \sqrt{\frac{\ell}{2\mu \ell + \mu \sigma^2}}, \text{ and } \overline{\sigma}_H \coloneqq \sqrt{2} (\ell + a + b + 1).\]
%In the special case where $\varepsilon_k = \varepsilon$ and $\varepsilon_k' = \varepsilon'$ for all $k$, we have
%\[\norm{(H^{-1} - (H')^{-1})_{i j}} \leq \frac{1}{\sigma_{min}^2} \left(\abs{i - j} + \frac{1}{1 - \rho}\right)\cdot \rho^{\abs{i - j}} \left(\varepsilon + \frac{2}{\rho}\cdot \varepsilon'\right).\]
\end{lemma}

\begin{proof}[Proof of \Cref{thm:diff-inverse-sophisticated}]
In the proof, we only show the results for $\Upsilon_t^T$. The results for $\hat{\Upsilon}_t^{t+k}$ can be shown using the same method.

We first show \eqref{thm:diff-inverse-sophisticated:e1} holds. By \Cref{lemma:block-singular-value}, we know that Note that $\Upsilon_{t}^{T}(\xi_{t:T})^2$ is a positive definite matrix that has band width $4$ and satisfies
\[\underline{\sigma}_H^2 I \preceq \Upsilon_{t}^{T}(\xi_{t:T})^2 \preceq \overline{\sigma}_H^2 I.\]
Using the same method as the proof of Lemma B.1 in \cite{lin2021perturbation}, one can show that for any $(i, j) \in [t, T]^2$,
\begin{align}\label{thm:diff-inverse-sophisticated:e3}
    \left(\left(\Upsilon_{t}^{T}(\xi_{t:T})^2\right)^{-1}\right)_{i j} \leq \frac{2}{\underline{\sigma}_H^2} \cdot \lambda_2^{\abs{i - j}}.
\end{align}
Note that $\Upsilon_{t}^{T}(\xi_{t:T})^{-1} \coloneqq \Upsilon_{t}^{T}(\xi_{t:T}) \cdot \left(\Upsilon_{t}^{T}(\xi_{t:T})^2\right)^{-1}$. Thus we see that
\begin{align*}
    \left(\Upsilon_{t}^{T}(\xi_{t:T})^{-1}\right)_{i j} &= \left(\Upsilon_{t}^{T}(\xi_{t:T}) \cdot \left(\Upsilon_{t}^{T}(\xi_{t:T})^2\right)^{-1}\right)_{i j}\\
    &= \sum_{k = t}^{T} \Upsilon_{t}^{T}(\xi_{t:T})_{i k} \cdot \left(\left(\Upsilon_{t}^{T}(\xi_{t:T})^2\right)^{-1}\right)_{k j}\\
    &= \sum_{k = i-1}^{i+1} \Upsilon_{t}^{T}(\xi_{t:T})_{i k} \cdot \left(\left(\Upsilon_{t}^{T}(\xi_{t:T})^2\right)^{-1}\right)_{k j}.
\end{align*}
Therefore, by \eqref{thm:diff-inverse-sophisticated:e3}, we see that
\begin{align}\label{thm:diff-inverse-sophisticated:e3-1}
    \norm{\left(\Upsilon_{t}^{T}(\xi_{t:T})^{-1}\right)_{i j}} \leq \frac{4(\ell + 1 + a + b)}{\underline{\sigma}_H^2 \cdot \lambda_2} \cdot \lambda_2^{\abs{i - j}}, \forall (i, j) \in [t, T]^2.
\end{align}

Note that we have
\begin{align}\label{thm:diff-inverse-sophisticated:e4}
    \Upsilon_{t}^{T}(\xi_{t:T})^{-1} - \Upsilon_{t}^{T}(\xi_{t:T}')^{-1} = - (\Upsilon_{t}^{T}(\xi_{t:T}'))^{-1} \left(\Upsilon_{t}^{T}(\xi_{t:T}) - \Upsilon_{t}^{T}(\xi_{t:T}')\right) \Upsilon_{t}^{T}(\xi_{t:T})^{-1}.
\end{align}
To simplify the notation, we define
\[E_\tau \coloneqq \left(\Upsilon_{t}^{T}(\xi_{t:T}) - \Upsilon_{t}^{T}(\xi_{t:T}')\right)_{\tau \tau}, \forall \tau \in [t, T],\]
and
\[E_\tau' \coloneqq \left(\Upsilon_{t}^{T}(\xi_{t:T}) - \Upsilon_{t}^{T}(\xi_{t:T}')\right)_{(\tau+1) \tau}, \forall \tau \in [t, T-1].\]
The right hand side of \eqref{thm:diff-inverse-sophisticated:e4} is a linear equation. Thus we can study the following 3 equations separately and sum them up:
\begin{subequations}\label{thm:diff-inverse-sophisticated:e5}
\begin{align}
    \Phi_\tau &\coloneqq (\Upsilon_{t}^{T}(\xi_{t:T}'))^{-1} \begin{bmatrix}
    0 & & & & & & \\
     & \ddots & & & & & \\
     & & 0 & & & & \\
     & & & E_\tau & & & \\
     & & & & 0 & & \\
     & & & & & \ddots & \\
     & & & & & & 0
    \end{bmatrix} (\Upsilon_{t}^{T}(\xi_{t:T}))^{-1}, \forall \tau \in [t, T],\label{thm:diff-inverse-sophisticated:e5:s1}\\
    \Phi_\tau^L &\coloneqq (\Upsilon_{t}^{T}(\xi_{t:T}'))^{-1} \begin{bmatrix}
    0 & & & & & \\
     & \ddots & & & & \\
     & & 0 & & & \\
     & & E_\tau' & 0 & & \\
     & & & & \ddots & \\
     & & & & & 0
    \end{bmatrix} (\Upsilon_{t}^{T}(\xi_{t:T}))^{-1}, \forall \tau \in [t, T-1],\label{thm:diff-inverse-sophisticated:e5:s2}\\
    \Phi_\tau^U &\coloneqq (\Upsilon_{t}^{T}(\xi_{t:T}'))^{-1} \begin{bmatrix}
    0 & & & & & \\
     & \ddots & & & & \\
     & & 0 & (E_\tau')^{\top} & & \\
     & & & 0 & & \\
     & & & & \ddots & \\
     & & & & & 0
    \end{bmatrix} (\Upsilon_{t}^{T}(\xi_{t:T}))^{-1}, \forall \tau \in [t, T-1].\label{thm:diff-inverse-sophisticated:e5:s3}
\end{align}
\end{subequations}
By \eqref{thm:diff-inverse-sophisticated:e5:s1}, \eqref{thm:diff-inverse-sophisticated:e5:s2}, and \eqref{thm:diff-inverse-sophisticated:e5:s3}, we see that
\begin{subequations}\label{thm:diff-inverse-sophisticated:e6}
\begin{align}
    \norm{(\Phi_\tau)_{i j}} &\leq \norm{\left((\Upsilon_{t}^{T}(\xi_{t:T}'))^{-1}\right)_{i \tau}} \cdot \norm{E_\tau} \cdot \norm{\left(\Upsilon_{t}^{T}(\xi_{t:T})^{-1}\right)_{\tau j}}\nonumber\\
    &\leq C_2^2 \cdot \lambda_2^{\abs{\tau - i} + \abs{\tau - j}} \cdot \norm{E_\tau},\label{thm:diff-inverse-sophisticated:e6:s1}\\
    \norm{(\Phi_\tau^L)_{i j}} &\leq \norm{\left((\Upsilon_{t}^{T}(\xi_{t:T}'))^{-1}\right)_{i \tau}} \cdot \norm{E_\tau'} \cdot \norm{\left(\Upsilon_{t}^{T}(\xi_{t:T})^{-1}\right)_{(\tau+1) j}}\nonumber\\
    &\leq C_2^2 \cdot \lambda_2^{\abs{\tau - i} + \abs{\tau+1 - j}} \cdot \norm{E_\tau'},\label{thm:diff-inverse-sophisticated:e6:s2}\\
    \norm{(\Phi_\tau^U)_{i j}} &\leq \norm{\left((\Upsilon_{t}^{T}(\xi_{t:T}'))^{-1}\right)_{i (\tau+1)}} \cdot \norm{E_\tau'} \cdot \norm{\left(\Upsilon_{t}^{T}(\xi_{t:T})^{-1}\right)_{\tau j}}\nonumber\\
    &\leq C_2^2 \cdot \lambda_2^{\abs{\tau+1 - i} + \abs{\tau - j}} \cdot \norm{E_\tau'},\label{thm:diff-inverse-sophisticated:e6:s3}
\end{align}
\end{subequations}
where we use the bound \eqref{thm:diff-inverse-sophisticated:e3-1} on the norm of individual block entries in \eqref{thm:diff-inverse-sophisticated:e6:s1}, \eqref{thm:diff-inverse-sophisticated:e6:s2}, and \eqref{thm:diff-inverse-sophisticated:e6:s3}. Summing these inequalities up over $\tau$, we see that
\begin{subequations}\label{thm:diff-inverse-sophisticated:e7}
\begin{align}
    &\norm{\left(\Upsilon_{t}^{T}(\xi_{t:T})^{-1} - \Upsilon_{t}^{T}(\xi_{t:T}')^{-1}\right)_{i j}}\nonumber\\
    ={}& \norm{\sum_{\tau=t}^{T} (\Phi_\tau)_{i j} + \sum_{\tau=t}^{T - 1} (\Phi_\tau^L)_{i j} + \sum_{\tau=t}^{T-1} (\Phi_\tau^U)_{i j}}\label{thm:diff-inverse-sophisticated:e7:s1}\\
    \leq{}& \sum_{\tau=t}^{T} \norm{(\Phi_\tau)_{i j}} + \sum_{\tau=t}^{T-1} \norm{(\Phi_\tau^L)_{i j}} + \sum_{\tau=t}^{T-1} \norm{(\Phi_\tau^U)_{i j}}\label{thm:diff-inverse-sophisticated:e7:s2}\\
    \leq{}& C_2^2 \left(\sum_{\tau=t}^{T} \lambda_2^{\abs{\tau - i} + \abs{\tau - j}} \cdot \norm{E_\tau} + \frac{2}{\lambda_2} \sum_{\tau=t}^{T-1} \lambda_2^{\abs{\tau - i} + \abs{\tau - j}} \cdot \norm{E_\tau'}\right)\label{thm:diff-inverse-sophisticated:e7:s3}\\
    \leq{}& C_2^2 \left(\max\{L_Q + L_R, L_P\} + \frac{2}{\lambda_2}(L_A + L_B)\right) \sum_{\tau=t}^{T} \lambda_2^{\abs{\tau - i} + \abs{\tau - j}} \cdot \norm{\xi_\tau - \xi_\tau'},\label{thm:diff-inverse-sophisticated:e7:s4}
\end{align}
\end{subequations}
where we use \eqref{thm:diff-inverse-sophisticated:e4} in \eqref{thm:diff-inverse-sophisticated:e7:s1}; we use the triangle inequality in \eqref{thm:diff-inverse-sophisticated:e7:s2}; we use \eqref{thm:diff-inverse-sophisticated:e6} in \eqref{thm:diff-inverse-sophisticated:e7:s3}; we use the Lipschitzness of dynamical and cost matrices in $\xi$ (Assumption \ref{assump:unconstrained:dynamics}) in \eqref{thm:diff-inverse-sophisticated:e7:s4}.
\end{proof}

With \Cref{thm:diff-inverse-sophisticated}, we can derive the perturbation bounds specified by Property \ref{assump:pipeline-perturbation-bounds}.

\begin{theorem}\label{thm:perturbation-bound-unconstrained-dynamics}
Under Assumption \ref{assump:unconstrained:dynamics}, Property \ref{assump:pipeline-perturbation-bounds} holds for arbitrary positive constant $R$ and $q_1(t) = H_2 \lambda_2^{2t}$, $q_2(t) = H_2 \lambda_2^t$, and $q_3(t) = H_2 \lambda_2^t$, where $\lambda_2$ is defined in \Cref{thm:diff-inverse-sophisticated}, and $H_2$ is given by
\[H_2 = C_2'\left(\frac{2 (\ell D_{\bar{x}} + D_w)}{1 - \lambda_2} + R + D_{x^*} + 1\right) + C_2 \left(L_w + \ell L_{\bar{x}} + D_{\bar{x}} L_Q + 1\right).\]
\end{theorem}
\begin{proof}[Proof of \Cref{thm:perturbation-bound-unconstrained-dynamics}]
For $t < T - k$, under the specification of Property \ref{assump:pipeline-perturbation-bounds}, we see that
\begin{subequations}\label{thm:perturbation-bound-unconstrained-dynamics:e1}
\begin{align}
    &\norm{\psi_t^{t+k}\left(z, \xi_{t:t+k}; \mathbb{I}\right)_{v_t} - \psi_t^{t+k}\left(z, \xi_{t:t+k}'; \mathbb{I}\right)_{v_t}}\nonumber\\
    \leq{}& \norm{\left(\hat{\Upsilon}_t^{t+k}(\xi_{t:t+k})^{-1} \hat{\beta}_t^{t+k}(z, \xi_{t:t+k}) - \hat{\Upsilon}_t^{t+k}(\xi_{t:t+k}')^{-1} \hat{\beta}_t^{t+k}(z, \xi_{t:t+k}')\right)_{v_t}} \label{thm:perturbation-bound-unconstrained-dynamics:e1:s1}\\
    \leq{}& \norm{\left(\left(\hat{\Upsilon}_t^{t+k}(\xi_{t:t+k})^{-1} - \hat{\Upsilon}_t^{t+k}(\xi_{t:t+k}')^{-1}\right) \hat{\beta}_t^{t+k}(z, \xi_{t:t+k})\right)_{v_t}}\nonumber\\
    &+ \norm{\left(\hat{\Upsilon}_t^{t+k}(\xi_{t:t+k}')^{-1} \left(\hat{\beta}_t^{t+k}(z, \xi_{t:t+k}) - \hat{\beta}_t^{t+k}(z, \xi_{t:t+k}')\right)\right)_{v_t}}, \label{thm:perturbation-bound-unconstrained-dynamics:e1:s2}
\end{align}
\end{subequations}
where we used the KKT condition in \eqref{thm:perturbation-bound-unconstrained-dynamics:e1:s1} and the triangle inequality in \eqref{thm:perturbation-bound-unconstrained-dynamics:e1:s2}.

For the first term in \eqref{thm:perturbation-bound-unconstrained-dynamics:e1:s2}, we see that
\begin{subequations}\label{thm:perturbation-bound-unconstrained-dynamics:e2}
\begin{align}
    &\norm{\left(\left(\hat{\Upsilon}_{t}^{t+k}(\xi_{t:t+k})^{-1} - \hat{\Upsilon}_{t}^{t+k}(\xi_{t:t+k}')^{-1}\right) \cdot \hat{\beta}_t^{t+k}(z, \xi_{t:t+k})\right)_{v_t}}\nonumber\\
    \leq{}& \sum_{\tau=t}^{t+k} \norm{\left(\hat{\Upsilon}_{t}^{t+k}(\xi_{t:t+k})^{-1} - \hat{\Upsilon}_{t}^{t+k}(\xi_{t:t+k}')^{-1}\right)_{t \tau}} \cdot \norm{\hat{\beta}_t^{t+k}(z, \xi_{t:t+k})_\tau} \label{thm:perturbation-bound-unconstrained-dynamics:e2:s1}\\
    \leq{}& C_2' \left(\sum_{\tau = 0}^k \lambda_2^{2\tau} \norm{\xi_{t+\tau} - \xi_{t+\tau}'}\right) \cdot \norm{z} + \sum_{\tau = t}^{t+k} C_2' \left(\sum_{i = t}^{t+k} \lambda_2^{i - t + \abs{i - \tau}} \norm{\xi_i - \xi_i'} \right)\cdot (\ell D_{\bar{x}} + D_w)\nonumber\\
    &+ C_2' \lambda_2^k\cdot \left(\sum_{\tau=0}^k \norm{\xi_{t+\tau} - \xi_{t+\tau}'}\right) \cdot \norm{\xi_{t+k}} \label{thm:perturbation-bound-unconstrained-dynamics:e2:s2}\\
    \leq{}& C_2' \sum_{\tau = 0}^k \lambda_2^{2\tau}\delta_{t+\tau} \cdot \norm{z} + C_2'\left(\frac{2 (\ell D_{\bar{x}} + D_w)}{1 - \lambda_2} + R + D_{x^*}\right)\sum_{\tau = 0}^k \lambda_2^\tau \delta_{t+\tau}, \label{thm:perturbation-bound-unconstrained-dynamics:e2:s3}
\end{align}
\end{subequations}
where we use the triangle inequality in \eqref{thm:perturbation-bound-unconstrained-dynamics:e2:s1}; we use \Cref{thm:diff-inverse-sophisticated} and the bounds on each entry of $\hat{\beta}_t^{t+k}(z, \xi_{t:t+k})$ in \eqref{thm:perturbation-bound-unconstrained-dynamics:e2:s2}; we rearrange the terms and use $\xi_{t+k} \in \mathcal{B}(x_{t+k}^*, R)$ in \eqref{thm:perturbation-bound-unconstrained-dynamics:e2:s3}.
For the second error term \eqref{thm:perturbation-bound-unconstrained-dynamics:e1:s2}, we see that
\begin{align}\label{thm:perturbation-bound-unconstrained-dynamics:e3}
    &\norm{\left(\hat{\Upsilon}_{t}^{t+k}(\xi_{t:t+k}')^{-1} \left(\hat{\beta}_t^{t+k}(z, \xi_{t:t+k}) - \hat{\beta}_t^{t+k}(z, \xi_{t:t+k}')\right)\right)_{v_t}}\nonumber\\
    \leq{}& C_2 \sum_{\tau = t}^{t+k} \lambda_2^{\tau - t} (L_w + \ell L_{\bar{x}} + D_{\bar{x}} L_Q) \delta_\tau + C_2 \lambda_2^k \delta_{t+k},
\end{align}
where we use the following inequality to bound the difference between $\hat{\beta}_t^{t+k}(z, \xi_{t:t+k})$ and $\hat{\beta}_t^{t+k}(z, \xi_{t:t+k}')$:
\begin{align*}
    &\norm{Q_\tau(\xi_\tau) \bar{x}_\tau(\xi_\tau) - Q_\tau(\xi_\tau') \bar{x}_\tau(\xi_\tau')}\\
    \leq{}& \norm{Q_\tau(\xi_\tau) \bar{x}_\tau(\xi_\tau) - Q_\tau(\xi_\tau') \bar{x}_\tau(\xi_\tau)} + \norm{Q_\tau(\xi_\tau') \bar{x}_\tau(\xi_\tau) - Q_\tau(\xi_\tau') \bar{x}_\tau(\xi_\tau')}\\
    \leq{}& \norm{Q_\tau(\xi_\tau) - Q_\tau(\xi_\tau')} \cdot \norm{\bar{x}_\tau(\xi_\tau)} + \norm{Q_\tau(\xi_\tau')} \cdot \norm{\bar{x}_\tau(\xi_\tau) - \bar{x}_\tau(\xi_\tau')}\\
    \leq{}& (L_Q D_{\bar{x}} + \ell L_{\bar{x}}) \delta_\tau.
\end{align*}
Substituting \eqref{thm:perturbation-bound-unconstrained-dynamics:e2} and \eqref{thm:perturbation-bound-unconstrained-dynamics:e3} into \eqref{thm:perturbation-bound-unconstrained-dynamics:e1} gives that for any $t < T - k$,
\[\norm{\psi_t^{t+k}\left(z, \xi_{t:t+k}; \mathbb{I}\right)_{v_t} - \psi_t^{t+k}\left(z, \xi_{t:t+k}'; \mathbb{I}\right)_{v_t}} \leq \left(\sum_{\tau=0}^{k} q_1(\tau) \delta_{t+\tau}\right) \norm{z} + \sum_{\tau = 0}^{k} q_2(\tau) \delta_{t+\tau}\]
under the specification that $\xi_{t:t+k-1} \in \varXi_{t:t+k-1}, \xi_{t:t+k-1}' = \xi_{t:t+k-1}^*;~ \xi_{t+k}, \xi_{t+k}' \in \mathcal{B}(x_{t+k}^*, R)$. We can use a similar methods to show that for any $t \geq T - k$,
\[\norm{\psi_t^{T}\left(z, \xi_{t:T}; F_T\right)_{v_t} - \psi_t^{T}\left(z, \xi_{t:T}'; F_T\right)_{v_t}} \leq \left(\sum_{\tau=0}^{T-t} q_1(\tau) \delta_{t+\tau}\right) \norm{z} + \sum_{\tau = 0}^{T-t} q_2(\tau) \delta_{t+\tau}\]
under the specification that $\xi_{t:T} \in \varXi_{t:T}, \xi_{t:T}' = \xi_{t:T}^*$.

For any $t < T$, we see that
\begin{subequations}\label{thm:perturbation-bound-unconstrained-dynamics:e4}
\begin{align}
    &\norm{\psi_t^T(z, \xi_{t:T}^*; F_T)_{y_\tau/v_\tau} - \psi_t^T(z', \xi_{t:T}^*; F_T)_{y_\tau/v_\tau}}\nonumber\\
    \leq{}& \norm{\left(\Upsilon_{t}^{T}(\xi_{t:T}^*)^{-1} \left(\beta_t^{T}(z, \xi_{t:T}^*) - \beta_t^T(z', \xi_{t:T}^*)\right)\right)_{y_\tau/v_\tau}}\label{thm:perturbation-bound-unconstrained-dynamics:e4:s1}\\
    \leq{}& \norm{\left(\Upsilon_{t}^{T}(\xi_{t:T}^*)^{-1}\right)_{\tau t}} \cdot \norm{z - z'}\nonumber\\
    \leq{}& C_2 \lambda_2^{\tau - t} \norm{z - z'},\label{thm:perturbation-bound-unconstrained-dynamics:e4:s2}
\end{align}
\end{subequations}
where we use the KKT condition in \eqref{thm:perturbation-bound-unconstrained-dynamics:e4:s1}; we use \Cref{thm:diff-inverse-sophisticated} in \eqref{thm:perturbation-bound-unconstrained-dynamics:e4:s2}.
\end{proof}

Now we come back to the proof of \Cref{thm:perturbation:unconstrained-LTV-pred-err-dynamics}.
\begin{proof}[Proof of \Cref{thm:perturbation:unconstrained-LTV-pred-err-dynamics}]
By \Cref{thm:perturbation-bound-unconstrained-dynamics}, Property \ref{assump:pipeline-perturbation-bounds} holds for arbitrary positive constant $R$ and $q_1(t) = H_2 \lambda_2^{2t}$, $q_2(t) = H_2 \lambda_2^t$, and $q_3(t) = H_2 \lambda_2^t$, where the decay rate $\lambda_2 \in (0, 1)$ and constant $H_2$ depends on $R$. We set $R \coloneqq D_{x^*} + D_{\bar{x}}$ so that $\MPC_k$ with terminal state $\bar{x}_{t+k}(\xi_{t+k\mid t})$ satisfies the assumption of \Cref{thm:the-pipeline-theorem}. The constant $H_2$ is given by
\[H_2 = C_2'\left(\frac{2 (\ell D_{\bar{x}} + D_w)}{1 - \lambda_2} + 2D_{x^*} + D_{\bar{x}} + 1\right) + C_2 \left(L_w + \ell L_{\bar{x}} + D_{\bar{x}} L_Q + 1\right).\]
By \Cref{thm:the-pipeline-theorem}, in order to achieve the claimed dynamic regret bound in \Cref{thm:perturbation:unconstrained-LTV-pred-err-dynamics}, a sufficient condition is that the prediction errors $\rho_{t, \tau}$ satisfy
\[\sum_{\tau = 0}^k \lambda_2^\tau \rho_{t, \tau} \leq \frac{(1 - \lambda_2)^2(D_{x^*} + D_{\bar{x}})}{2 H_2^2 L_g \left((1 - \lambda_2)(D_{x^*} + D_{\bar{x}}) + H_2(D_{x^*} + 1)\right)},\]
and the prediction horizon $k$ satisfies that
\[\lambda_2^k \leq \frac{(1 - \lambda_2)^2}{4 H_2^2 L_g \left((1 - \lambda_2)(D_{x^*} + D_{\bar{x}}) + H_2(D_{x^*} + 1)\right)}.\]
\end{proof}

{
We provide two example systems where the dynamics is unknown but the uniform controllability assumption (Assumption \ref{assump:unconstrained-LTV-pred-err-dynamics}) is satisfied for all possible uncertainty parameters.

\begin{example}[Inverted pendulum with unknown mass]\label{example:inverted-pendulum}
We consider the linearized inverted pendulum dynamics in discrete state space. The dynamics of the system is given by
\begin{align*}
    \begin{bmatrix}
    x_{t+1}\\
    \dot{x}_{t+1}\\
    \phi_{t+1}\\
    \dot{\phi}_{t+1}
    \end{bmatrix} = \underbrace{\begin{bmatrix}
    1 & \delta & 0 & 0\\
    0 & 1 + \frac{-(I + m l^2) b\delta}{I(M + m) + M m l^2} & \frac{m^2 g l^2 \delta}{I(M + m) + M m l^2} & 0\\
    0 & 0 & 1 & \delta\\
    0 & \frac{- m l b \delta}{I(M + m) + M m l^2} & \frac{m g l (M + m)\delta}{I(M + m) + M m l^2} & 1
    \end{bmatrix}}_{A(M)} \begin{bmatrix}
    x_{t}\\
    \dot{x}_{t}\\
    \phi_{t}\\
    \dot{\phi}_{t}
    \end{bmatrix} + \underbrace{\begin{bmatrix}
    0\\
    \frac{(I + m l^2)\delta}{I(M + m) + M m l^2}\\
    0\\
    \frac{ml\delta}{I(M + m) + M m l^2}
    \end{bmatrix}}_{B(M)} u_t,
\end{align*}
where $M$ is the mass of the cart, $m$ is the mass of the pendulum, $b$ is the coefficient of friction for cart, $l$ is the length to the center of the mass of the pendulum, $I$ is the mass moment of inertia of the pendulum, $x_t$ is the position of the cart, $\phi_y$ is the angle of the pendulum, and $\delta$ is step size of discretization. For simplicity, we assume all system parameters are known except the mass of the cart $M$, which is an unknown value in the interval $\left[\underline{M}, \overline{M}\right]$ for some constants $\overline{M} > \underline{M} > 0$. The controllability matrix of the dynamical system satisfies that
\begin{align*}
    \det\left[B(M), A(M) B(M), A(M)^2 B(M), A^3(M) B(M)\right] &= \frac{\delta^{10}g^2 l^4 m^4}{(I M + m (I + l^2 M))^4}\\
    &\geq \frac{\delta^{10}g^2 l^4 m^4}{(I \overline{M} + m (I + l^2 \overline{M}))^4} > 0.
\end{align*}
Therefore, we can use Theorem 1 in \cite{hong1992lower} to show
\[\sigma_{min}\left[B(M), A(M) B(M), A(M)^2 B(M), A^3(M) B(M)\right] \geq \sigma_0\]
for some positive constant $\sigma_0$, which implies Assumption \ref{assump:unconstrained-LTV-pred-err-dynamics} by Lemma 12 in \cite{shin2021controllability}.
\end{example}

\begin{example}[Frequency regulation with unknown inertia]\label{example:frequency-regulation}
We consider the power grid dynamics with $n$ nodes studied in \cite{gonzalez2019powergrid}:
\[\underbrace{\begin{bmatrix}
        \dot{\theta} \\ \dot{\omega}
      \end{bmatrix}}_{\dot{x}(t)}
      = \underbrace{\begin{bmatrix}
        0 & I \\
        -M_{q(t)}^{-1}L & -M_{q(t)}^{-1}D
      \end{bmatrix}}_{\hat{A}(t)} 
      \underbrace{\begin{bmatrix}
        \theta \\ \omega
      \end{bmatrix}}_{x(t)}
      + \underbrace{\begin{bmatrix}
        0 \\ M_{q(t)}^{-1}
      \end{bmatrix}}_{\hat{B}(t)}
      \underbrace{p_{\text{in}}}_{u(t)}. \]
Here, $\theta, \omega \in \mathbb{R}^n$ are the vectors of voltage phase angles and frequencies. The state $x(t)$ is a stacked vector of $\theta$ and $\omega$, and $u(t) = p_{in}$ is the power input. $D$ is a diagonal matrix whose entries represent droop control coefficients, and $L$ is the Laplacian matrix of the network of nodes. $M_{q(t)}$ is the inertia matrix in mode $q(t) \in \{1, \ldots, m\}$ and is time-varying. It is further assumed to be in the form $M_{q(t)} = m_{q(t)} I$, where $m_{q(t)}$ is a scalar that represents the inertia coefficient at time $t$. We assume all system parameters are known except the inertia coefficient, and we define $\xi_t \coloneqq m_{q(t)}$.

For simplicity, we use a different discretization technique with \cite{gonzalez2019powergrid} so that it is easier to verify uniform controllability. Specifically, we write the discrete-time system as
\[x_{t+1} = \underbrace{(I + \delta \hat{A}(t))}_{A(\xi_t)} x_t + \underbrace{\delta \hat{B}(t)}_{B(\xi_t)} u_t,\]
where $\delta$ is the step size of discretization. We see that when $0 < \underline{m} \leq m_{q(t)} \leq \overline{m} < +\infty$ holds for some positive constants $\underline{m}, \overline{m}$ for all possible inertia coefficient, we have
\[\abs{\det\left[B(\xi_{t+1}), A(\xi_{t+1}) B(\xi_t)\right]} \geq \frac{\delta^{3n}}{\overline{m}^{2n}} > 0.\]
Thus, we can use Theorem 1 in \cite{hong1992lower} to show that
\[\sigma_{min}\left[B(\xi_{t+1}), A(\xi_{t+1}) B(\xi_t)\right] \geq \sigma_0\]
holds for all possible inertia coefficient at all time steps for some positive constant $\sigma_0$, which implies Assumption \ref{assump:unconstrained-LTV-pred-err-dynamics} by Lemma 12 in \cite{shin2021controllability}.
\end{example}
}

\section{Assumptions and Proofs of Section \ref{sec:general}}\label{appendix:general}
To introduce the SSOSC assumption, we first define the \textit{reduced Hessian} of the Lagrangian.

\begin{definition}[reduced Hessian]
   For a constrained optimization problem with primal variable $z$ and dual variable $\eta$, let $H = \nabla^2_{zz} \mathcal{L}$ denote the Hessian of the Lagrangian $\mathcal{L}(z, \eta; \xi)$. Let $G$ denote the \textbf{active constraints Jacobian}, i.e. Jacobian of all equality constraints and active inequality constraints, and let $Z$ be the null-space matrix of $G$ (i.e., the column vectors of $Z$ form an orthonormal basis of the null space of $G$). Then the \textbf{reduced Hessian} is defined as $\Hre(z, \eta; \xi) := Z^{\top} H Z$.
\end{definition}

We define the concept of \textit{singular spectrum bounds} for a specific instance of FTOCP:

\begin{definition}[singular spectrum bounds]\label{def:controllability-tuple}
Consider the FTOCP $\iota_{t_1}^{t_2}(z, \xi_{t_1:t_2}; F)$. The positive real numbers $\overline{\sigma}_H, \overline{\sigma}_R, \underline{\sigma}_{H}$ are called \textbf{singular spectrum bounds} for this specific instance of FTOCP\footnote{We remind the reader that the functions $\overline{\overline{\sigma}}_H, \overline{\overline{\sigma}}_R,$ and $\underline{\underline{\sigma}}_H$ depend on the form of FTOCP, i.e.,  different horizon $[t_1, t_2]$ and different terminal cost function $F$.} if they satisfy that
\[\overline{\sigma}_H \geq \overline{\overline{\sigma}}_H(z, \xi_{t_1:t_2}), \overline{\sigma}_R \geq \overline{\overline{\sigma}}_R(z, \xi_{t_1:t_2}), \text{ and } 0 < \underline{\sigma}_H \leq \underline{\underline{\sigma}}_H(z, \xi_{t_1:t_2}),\]
where $\overline{\overline{\sigma}}_H, \overline{\overline{\sigma}}_R,$ and $\underline{\underline{\sigma}}_H$ are defined in (4.16a-c) in \cite{shin2021exponential}.
\end{definition}

\begin{assumption}\label{assump:general-dynamical-system}
We make the following assumptions on the costs, dynamics, and constraints of an FTOCP $\iota_{t_1}^{t_2}(z, \xi_{t_1:t_2}; F)$:
\begin{enumerate}
    \item All cost functions, dynamical functions, and constraint functions are twice continuously differentiable in $(x_t, u_t)$ and $\xi_t$ \footnote{If the terminal function $F$ is an indicator function of some state, we view it as a constraint instead of cost.}.
    \item (SSOSC) The reduced Hessian at the optimal primal-dual solution is positive-definite.
    \item (LICQ) The active constraints Jacobian $G$ at the optimal primal-dual solution has full row rank, i.e. $\sigma_{\min}(G) > 0$.
    \item (Uniform singular spectrum bounds) There exist positive singular spectrum bounds $\overline{\sigma}_H, \overline{\sigma}_R, \underline{\sigma}_{H}$ for all FTOCP specifications below:
    \begin{enumerate}
        \item $t_1 = t, t_2 = t + k$ for $t < T - k$:
        \[z \in \mathcal{B}(x_t^*, R), \xi_{t:t+k-1} \in \varXi_{t:t+k-1}, \xi_{t+k} \in \mathcal{B}(x_{t+k}^*, R), F = \mathbb{I}.\]
        \item $t_1 = t, t_2 = T$ for $t < T$:
        \[z \in \mathcal{B}(x_t^*, R), \xi_{t:T} \in \varXi_{t:T}, F = F_T.\]
    \end{enumerate}
\end{enumerate}
\end{assumption}

%\yang{do we need to make the assumption for a neighborhood around $\OPT$?} \yiheng{I think making the assumption for a neighborhood around the optimal trajectory will be better.}

We remind the readers that Lemma 12 in \cite{shin2021controllability} shows that Lipshitzness of dynamics and uniform controllability together imply uniform LICQ property of the system.

Under Assumption \ref{assump:general-dynamical-system}, we know Property \ref{assump:pipeline-perturbation-bounds} holds for $q_1(t) = 0$, $q_2(t) = H_3 \lambda_3^t$, and $q_3(t) = H_3 \lambda_3^t$ for some $H_3 > 0$ and $\lambda_3 \in (0, 1)$ by Theorem 4.5 in \cite{shin2021exponential}.

\begin{theorem}\label{thm:perturbation-bound-general-system}
Under Assumption \ref{assump:general-dynamical-system} that holds for some $R > 0$, Property \ref{assump:pipeline-perturbation-bounds} holds for $q_1(t) = 0$, $q_2(t) = H_3 \lambda_3^t$, and $q_3(t) = H_3 \lambda_3^t$ with the same $R$. The coefficient $H_3$ and decay factor $\lambda_3$ are given by
\[H_3 \coloneqq \left(\frac{\overline{\sigma}_H \overline{\sigma}_R}{\underline{\sigma}_H^2}\right)^{\frac{1}{2}}, \text{ and } \lambda_3 \coloneqq \left(\frac{\overline{\sigma}_H^2 - \underline{\sigma}_H^2}{\overline{\sigma}_H^2 + \underline{\sigma}_H^2}\right)^{\frac{1}{8}}.\]
\end{theorem}

Combining \Cref{thm:perturbation-bound-general-system} with the Pipeline Theorem (\Cref{thm:the-pipeline-theorem}) finishes the proof of \Cref{thm:perturbation:general-system}. Note that when $\xi_t^* \leq \frac{(1 - \lambda_3)R}{H_3}$, we know that $\norm{\psi_0^T\left(x_0, \mathbf{0}; F_T\right)_{y_{t+k}} - x_{t+k}^*} \leq R$. Thus using $\psi_0^T\left(x_0, \mathbf{0}; F_T\right)_{y_{t+k}}$ as the terminal state of $\MPC_k$ at time step $t$ can satisfy the requirement of \Cref{thm:perturbation:general-system}.

%\section{(obsolete) Assumptions and Proofs of Constrained LQR Systems }\label{appendix:general-constrained-proofs}
%\input{Proofs/general-constrained}

%\section{(obsolete) Assumptions and Proofs of Nonlinear Systems}\label{appendix:general-nonlinear-proofs}
%\input{Proofs/general-nonlinear}

\section{Inventory control with constraints}\label{appendix:inventory-control}
In this section, we present two examples about the perturbation bounds under a simple inventory control dynamics. In the first example, we use the results in \cite{shin2021exponential} to show a general exponentially decaying perturbation bound holds when we only have one-sided constraints on control input $u_t$. 

\begin{theorem}\label{thm:exp-decay-one-side-inventory}
Consider the optimal control problem where the state $x_t$ and the control input $u_t$ are both in $\mathbb{R}$. The dynamics and constraints are given by
\[x_{t+1} = x_t + u_t,\text{ s.t. } x_t \in [-1, 1], u_t \geq - \frac{4}{5}\]
for all $t$. The stage cost is given by $f_t(x_t, u_t; \xi_t)$, where $f_t$ is convex and $\ell$-smooth. We also assume that $f_t$ is $\mu$-strongly convex in its first variable $x_t$. Then, for any positive integer $p \geq 3$, we have
\begin{align}\label{thm:exp-decay-one-side-inventory:e1}
    &\abs{\psi_0^p\left(x_0, \xi_{0:p-1}, \zeta_p ; \mathbb{I}\right)_{x_h} - \psi_0^p\left(x_0', \xi_{0:p-1}', \zeta_p' ; \mathbb{I}\right)_{x_h}}\nonumber\\
    \leq{}& C \left(\lambda^h \abs{x_0 - x_0'} + \sum_{\tau = 0}^{p-1} \lambda^{\abs{h - \tau}} \abs{\xi_\tau - \xi_\tau'} + \lambda^{p-h} \abs{\zeta_p - \zeta_p'}\right),
\end{align}
for all $x_0, \zeta_p \in [-1, 1]$, where $h \in \{1, \ldots, p\}$ and $C > 0, \lambda \in (0, 1)$ are some constants.
\end{theorem}
\begin{proof}[Proof of \Cref{thm:exp-decay-one-side-inventory}]
We can rewrite the optimization problem to remove the equality constraints as following:
\begin{subequations}\label{thm:exp-decay-one-side-inventory:e2}
\begin{align}
    \min_{x_{1:p-1}}& \sum_{t=0}^{p-1} f_t(x_t, x_{t+1} - x_{t}; \xi_t)\\
    \text{ s.t. }& -1 \leq x_t \leq 1, \forall t \in \{1, 2, \ldots, p-1\},\\
    & x_t - x_{t-1} \geq - \frac{4}{5}, \forall t \in \{1, 2, \ldots, p\},
\end{align}
\end{subequations}
where $x_p = \zeta_p$.

Note that for any time index $t \in \{1, 2, \ldots, p - 1\}$, at most 2 constraints that involves $x_t$ can be active. They can be chosen from the 4 possible constraints that involves $x_t$:
\[x_t \geq -1, x_t \leq 1, x_{t} - x_{t-1} \geq -\frac{4}{5}, x_{t+1} - x_t \geq -\frac{4}{5}.\]
And for any time index $t \in \{1, 2, \ldots, p - 2\}$, the 3 consecutive ``coupling'' constraints
\begin{equation}\label{thm:exp-decay-one-side-inventory:e3}
    x_{t} - x_{t-1} \geq -\frac{4}{5}, x_{t+1} - x_t \geq -\frac{4}{5}, x_{t+2} - x_{t+1} \geq -\frac{4}{5},
\end{equation}
cannot activate simultaneously. Let $\sigma_0$ denote the smallest singular value of matrix
\[\begin{bmatrix}
1 & & \\
-1 & 1& \\
 & -1 & 1
\end{bmatrix}.\]
Therefore, in the context of Theorem 4.5 in \cite{shin2021exponential}, we see that $\underline{\sigma}\left(\nabla_{xy} \mathcal{L}\left(z^\dagger (\xi); \xi\right)[\mathcal{B}, \mathcal{B}]\right)$ is lower bounded by $\sigma_0$ and upper bounded by $2$. Since we also have that
\[\mu I \preceq \nabla_{xx} \mathcal{L}\left(z^\dagger (\xi); \xi\right)[\mathcal{B}, \mathcal{B}] \preceq 5\ell I.\]
By \Cref{lemma:block-singular-value}, we further see that we can set $\underline{\sigma}_H$ and $\overline{\sigma}_H$ as
\[\underline{\sigma}_H \coloneqq 2 \min(\mu, 1) \sqrt{\frac{5\ell}{10 \mu \ell + \mu \sigma_0^2}}, \overline{\sigma}_H \coloneqq \sqrt{2}(5\ell + 2).\]
We can set $\overline{\sigma}_R \coloneqq \ell$. Applying Theorem 4.5 in \cite{shin2021exponential} finishes the proof.
\end{proof}

As a remark, we have already certified Assumption \ref{assump:general-dynamical-system} in the proof of \Cref{thm:exp-decay-one-side-inventory} and the perturbation bound provided by \Cref{thm:exp-decay-one-side-inventory} is more general than the statement of \Cref{thm:perturbation-bound-general-system}.

While \Cref{thm:exp-decay-one-side-inventory} shows that Assumption \ref{assump:general-dynamical-system} is not vacuous, and we present a negative result in \Cref{thm:exp-decay-impossible} which shows exponentially decaying perturbation bounds may not hold when Assumption \ref{assump:general-dynamical-system} is not satisfied.

\begin{theorem}\label{thm:exp-decay-impossible}
Consider the optimal control problem where the state $x_t$ and the control input $u_t$ are both in $\mathbb{R}$. The dynamics and constraints are given by
\[x_{t+1} = x_t + u_t,\text{ s.t. } x_t \in [-1, 1], u_t \in \left[-\frac{4}{5}, \frac{4}{5}\right]\]
for all $t$. The stage cost is given by $f_t(x_t, u_t; \xi_t) = (x_t - \xi_t)^2$, where $\xi_t = \frac{4}{5}$ if $t$ is odd, and $\xi_t = -\frac{4}{5}$ if $t$ is even. For any $p$ is even, we have
\begin{equation}\label{thm:exp-decay-impossible:e1}
    \abs{\psi_{0}^p\left(0, \xi_{0:p-1}, -\frac{2}{5}; \mathbb{I}\right)_{x_h} - \psi_{0}^p\left(0, \xi_{0:p-1}, -\frac{2}{5} + \epsilon; \mathbb{I}\right)_{x_h}} = \epsilon
\end{equation}
holds for any $\epsilon \in [0, \frac{2}{5(p - 1)}]$ and $h \in \{1, \ldots, p\}$. For any $p$ is odd, we have
\begin{equation}\label{thm:exp-decay-impossible:e2}
    \abs{\psi_{0}^p\left(0, \xi_{0:p-1}, \frac{2}{5}; \mathbb{I}\right)_{x_h} - \psi_{0}^p\left(0, \xi_{0:p-1}, \frac{2}{5} + \epsilon; \mathbb{I}\right)_{x_h}} = \epsilon
\end{equation}
holds for any $\epsilon \in \left[0, \frac{2}{5p}\right]$ and $h \in \{1, \ldots, p\}$.
\end{theorem}

Before presenting the proof of \Cref{thm:exp-decay-impossible}, we want to add a remark about why a similar proof as \Cref{thm:exp-decay-one-side-inventory} cannot work here. Note that a key property we leveraged in the proof of \Cref{thm:exp-decay-one-side-inventory} is that any three consecutive ``coupling'' constraints \eqref{thm:exp-decay-one-side-inventory:e3} cannot activate simultaneously. This is no longer the case when $u_t$ has two sides of constraints, i.e., $u_t \in \left[-\frac{4}{5}, \frac{4}{5}\right]$. As a result, the smallest singular value of matrix $\nabla_{xy} \mathcal{L}\left(z^\dagger (\xi); \xi\right)[\mathcal{B}, \mathcal{B}]$ can be arbitrarily small (i.e., decaying w.r.t. the horizon length $p$). Thus, Assumption \ref{assump:general-dynamical-system} is not satisfied because $\underline{\sigma}_H$ cannot be set as a positive constant, and the same proof as \Cref{thm:exp-decay-one-side-inventory} can no longer work. We will leverage this intuition to construct a counterexample to show \Cref{thm:exp-decay-impossible}: We construct a sequence of cost functions so that $u_t$ reaches either its lower bound $-\frac{4}{5}$ or its upper bound $\frac{4}{5}$ at every time step.

\begin{proof}[Proof of \Cref{thm:exp-decay-impossible}]
We first show that \eqref{thm:exp-decay-impossible:e1} holds by induction on $p$. Specifically, we will show that the following holds for any $q \in \mathbb{Z}_+$
\begin{equation}\label{thm:exp-decay-impossible:e3}
    \iota_0^{2q}\left(0, \xi_{0:2q-1}, -\frac{2}{5} + \epsilon; \mathbb{I}\right) \begin{cases}
    = \frac{8(q+2)}{25} + 2q \epsilon^2 &\text{ if } \epsilon \in [0, \frac{2}{5(2q - 1)}],\\
    \geq \frac{8(q+2)}{25} + \frac{8q}{25(2q - 1)^2} &\text{ if } \epsilon \in (\frac{2}{5(2q - 1)}, \frac{7}{5}],\\
    \geq \frac{8(q+2)}{25} &\text{ if } \epsilon \in [-\frac{3}{5}, 0),
    \end{cases}
\end{equation}
by induction on $q$.

It is straightforward to check that $\eqref{thm:exp-decay-impossible:e3}$ holds for $q = 1$. Suppose it holds for $q$. For $q+1$, we consider the following three cases separately:

\textbf{Case 1:} $0 \leq \epsilon \leq \frac{2}{5(2q + 1)}$.

Suppose $x_{2q} = -\frac{2}{5} + \epsilon'$. When $\epsilon' \in [0, \epsilon]$, we should choose $x_{2q+1} = \frac{2}{5} + \epsilon'$ to minimize the total cost. The total cost is given by
\[\frac{8(q+2)}{25} + 2q (\epsilon')^2 + \left(\frac{2}{5} - \epsilon'\right)^2 + \left(\frac{2}{5} + \epsilon\right)^2,\]
and it is minimized at $\epsilon' = \epsilon$. Thus, we achieve the total cost of $\frac{8(q+3)}{25} + 2(q+1)\epsilon^2$. When $\epsilon' > \epsilon$, note that the optimal choice of $x_{2q+1}$ is $\frac{2}{5} + \epsilon$, which is the same as when $\epsilon' = \epsilon$. By the induction assumption on $\iota_0^{2q}$, we see that the total cost incurred is lower bounded by $\frac{8(q+3)}{25} + 2(q+1)\epsilon^2$. When $\epsilon' < 0$, we have $x_{2q+1} \leq \frac{2}{5}$. Therefore, by the induction assumption, the total cost is lower bounded by
\[\frac{8(q + 2)}{25} + \frac{4}{25} + \left(\frac{2}{5} + \epsilon\right)^2 = \frac{8(q+3)}{25} + \epsilon^2 + \frac{4}{5}\epsilon \geq \frac{8(q+1)}{25} + 2(q+1)\epsilon^2.\]
Thus, we have shown that
\[\iota_0^{2q}\left(0, v_{0:2q-1}, -\frac{2}{5} + \epsilon; \mathbb{I}\right) = \frac{8(q+3)}{25} + 2(q+1) \epsilon^2, \forall \epsilon \in \left[0, \frac{2}{5(2q + 1)}\right].\]

\textbf{Case 2:} $\frac{2}{5(2q + 1)} < \epsilon \leq \frac{7}{5}$.

Suppose $x_{2q} = -\frac{2}{5} + \epsilon'$. When $\epsilon' \leq \frac{2}{5}$, we know that $x_{2q+1} \leq \frac{2}{5} + \epsilon' \leq \frac{4}{5}$. The total cost lower bounded by
\[\frac{8(q+2)}{25} + 2q (\epsilon')^2 + \left(\frac{2}{5} - \epsilon'\right)^2 + \left(\frac{2}{5} + \epsilon\right)^2.\]
Note that 
\[\frac{8(q+2)}{25} + 2q (\epsilon')^2 + \left(\frac{2}{5} - \epsilon'\right)^2 \geq \frac{8(q+2)}{25} + 2q \left(\frac{2}{5(2q+1)}\right)^2 + \left(\frac{2}{5} - \frac{2}{5(2q+1)}\right)^2.\]
Thus the total cost is lower bounded by 
\[\frac{8(q+3)}{25} + \frac{8(q + 1)}{25(2q + 1)^2}.\]
When $\epsilon' > \frac{2}{5}$, we see that the total cost is lower bounded by
\begin{align*}
    &\frac{8(q+2)}{25} + 2q (\epsilon')^2 + \left(\frac{2}{5} + \epsilon\right)^2\\
    \geq{}& \frac{8(q+2)}{25} + 2q \left(\frac{2}{5(2q+1)}\right)^2 + \left(\frac{2}{5} - \frac{2}{5(2q+1)}\right)^2 + \left(\frac{2}{5} + \frac{2}{5(2q+1)}\right)^2\\
    ={}& \frac{8(q+3)}{25} + \frac{8(q + 1)}{25(2q + 1)^2}.
\end{align*}

\textbf{Case 3:} $-\frac{3}{5} \leq \epsilon < 0$.

Note that the total cost of steps $0$ to $2q$ is uniformly lower bounded by $\frac{8(q+2)}{25}$ regardless of the choice of $x_{2q}$, and the total cost of steps $(2q+1)$ and $(2q+2)$ is uniformly lower bounded by $\frac{8}{25}$. Therefore, we see that
\[\iota_0^{2(q+1)}\left(0, \xi_{0:2q+1}, -\frac{2}{5} + \epsilon; \mathbb{I}\right) \geq \frac{8(q+3)}{25}.\]

Therefore, by combining the three cases, we have shown that \eqref{thm:exp-decay-impossible:e3} holds for all $q$ by induction. 

By rolling out the optimal states that minimize the total cost, one can show the unique optimal solution is given by
\begin{align*}
    \psi_{0}^p\left(0, \xi_{0:p-1}, -\frac{2}{5} + \epsilon; \mathbb{I}\right)_{x_h} = \begin{cases}
        \frac{2}{5} + \epsilon & \text{ if } h \text{ is odd,}\\
        -\frac{2}{5} + \epsilon & \text{ if } h \text{ is even,}
    \end{cases}
\end{align*}
when $\epsilon \in [0, \frac{2}{5(2q - 1)}]$. This finishes the proof of \eqref{thm:exp-decay-impossible:e1}.

For \eqref{thm:exp-decay-impossible:e2}, suppose $p = 2q + 1$. It is straightforward to verify that \eqref{thm:exp-decay-impossible:e2} holds for $q = 0$. When $q \geq 1$, by \eqref{thm:exp-decay-impossible:e3}, we know that in the optimal solution, we have $x_{2q} = -\frac{2}{5} + \epsilon$. We can further derive that the unique optimal solution is given by
\begin{align*}
    \psi_{0}^p\left(0, \xi_{0:p-1}, -\frac{2}{5} + \epsilon; \mathbb{I}\right)_{x_h} = \begin{cases}
        \frac{2}{5} + \epsilon & \text{ if } h \text{ is odd,}\\
        -\frac{2}{5} + \epsilon & \text{ if } h \text{ is even,}
    \end{cases}
\end{align*}
when $\epsilon \in [0, \frac{2}{5(2q + 1)}]$. This finishes the proof of \eqref{thm:exp-decay-impossible:e2}.
\end{proof}

\end{document}